\documentclass{article}
\usepackage{amsmath, amssymb}
\usepackage[all]{xy}
\newtheorem{thm}{Theorem}[section]
\newtheorem{cor}[thm]{Corollary}
\newtheorem{prop}[thm]{Proposition}
\newtheorem{lem}[thm]{Lemma}
\newenvironment{dfn}{\medskip\refstepcounter{thm}
\noindent{\bf Definition \thesection.\arabic{thm}\ }}{\medskip}
\newenvironment{ex}{\medskip\refstepcounter{thm}
\noindent{\bf Example \thesection.\arabic{thm}\ }}{\medskip}
\newenvironment{proof}[1][,]{\medskip\ifcat,#1
\noindent{{\it Proof}:\ }\else\noindent{\it Proof of #1.\ }\fi}
{\hfill$\square$\medskip}
\newenvironment{remark}[1][Remark]{\begin{trivlist}
\item[\hskip \labelsep {\bfseries #1}]}{\end{trivlist}}
\newenvironment{remarks}[1][Remarks]{\begin{trivlist}
\item[\hskip \labelsep {\bfseries #1}]}{\end{trivlist}}
\newenvironment{note}[1][Note]{\begin{trivlist}
\item[\hskip \labelsep {\bfseries #1}]}{\end{trivlist}}
\newenvironment{notes}[1][Notes]{\begin{trivlist}
\item[\hskip \labelsep {\bfseries #1}]}{\end{trivlist}}
\newenvironment{ack}[1][Acknowledgements]{\begin{trivlist}
\item[\hskip \labelsep {\bfseries #1}]}{\end{trivlist}}

\newenvironment{qnA}[1][Question A]{\begin{trivlist}
\item[\hskip \labelsep {\bfseries #1}]}{\end{trivlist}}
\newenvironment{qnB}[1][Question B]{\begin{trivlist}
\item[\hskip \labelsep {\bfseries #1}]}{\end{trivlist}}
\newenvironment{qnC}[1][Question C]{\begin{trivlist}
\item[\hskip \labelsep {\bfseries #1}]}{\end{trivlist}}

\def\O{{\mathbb O}}
\def\C{{\mathbb C}}
\def\H{{\mathbb H}}
\def\Im{\mathop{\rm Im}\nolimits} 
\def\Re{\mathop{\rm Re}\nolimits} 
\def\R{{\mathbb R}}
\DeclareMathOperator\Spin{Spin}
\DeclareMathOperator\SU{SU}

\DeclareMathOperator\GG{G}
\DeclareMathOperator\II{II}
\DeclareMathOperator\III{III}
\def\spin7{\mathfrak{spin}(7)}
\def\bfe{{\mathbf e}}
\def\bff{{\mathbf f}}

\def\g2{{\mathfrak g}_2}
\def\bfx{{\mathbf x}}
\def\bfu{{\mathbf u}}
\def\bfv{{\mathbf v}}
\def\bft{{\mathbf t}}
\def\bfn{{\mathbf n}}
\def\bfb{{\mathbf b}}

\def\P{{\mathbb P}}

\DeclareMathOperator\U{U}
\DeclareMathOperator\GL{GL}

\def\gg{{\rm g}}
\def\hh{{\rm h}}
\def\d{{\rm d}}
\def\w{\wedge}
\def\eq#1{{\rm(\ref{#1})}}
\DeclareMathOperator\SO{SO}
\def\Z{{\mathbb Z}}

\DeclareMathOperator\Tr{Tr}

\DeclareMathOperator\vol{vol}
\DeclareMathOperator\Gr{Gr}
\DeclareMathOperator\Id{Id}
\DeclareMathOperator\End{End}
\DeclareMathOperator\MM{M}

\begin{document}

\title{Associative Submanifolds of the 7-Sphere}
\author{Jason D. Lotay\footnote{Address for correspondence: Department of Mathematics, University College London, Gower Street, London WC1E 6BT, England, U.K. Email: j.lotay@ucl.ac.uk.}\\ {\normalsize Imperial College London}}

\date{}

\maketitle

\begin{abstract}
\noindent Associative submanifolds of the 7-sphere $\mathcal{S}^7$ are 3-dimensional minimal submanifolds which 
are the links of calibrated 4-dimensional cones in $\R^8$ called Cayley cones.  Examples of associative 3-folds are thus 
given by the links of complex and special Lagrangian cones in $\C^4$, as well as Lagrangian submanifolds of the 
nearly K\"ahler 6-sphere.  By classifying the associative group 
orbits, we exhibit the first known explicit example of an associative 3-fold in $\mathcal{S}^7$ which does not arise from 
 other geometries. 
 We then study associative 3-folds satisfying the curvature constraint known as Chen's equality, 
which is equivalent to a natural pointwise condition on the second fundamental form, and describe them 
using a new family of pseudoholomorphic curves in the Grassmannian of 2-planes in $\R^8$ and isotropic
 minimal surfaces in $\mathcal{S}^6$.   We also prove that 
 associative 3-folds which are ruled by geodesic circles, like minimal surfaces in space forms, admit families of local isometric deformations.  
 Finally, we construct associative 3-folds satisfying Chen's equality which have an 
 $\mathcal{S}^1$-family of global isometric deformations using harmonic 2-spheres in $\mathcal{S}^6$.
\end{abstract}

\section[Introduction]{Introduction}\label{intro}



\noindent Associative submanifolds of the 7-sphere $\mathcal{S}^7$ are 3-dimensional \emph{minimal} submanifolds 
  which are related to the $\GG_2$ structure that $\mathcal{S}^7$ naturally inherits from the standard $\Spin(7)$ structure 
on $\R^8$.  
Since $\GG_2$ acts irreducibly on $\R^7$, we are considering $\mathcal{S}^7$ with its standard constant curvature 1 metric.   
Associative submanifolds were introduced by Harvey and Lawson 
in \cite{HarLaw} in their seminal work on \emph{calibrated geometries} as 3-dimensional submanifolds of 7-manifolds endowed
 with a $\GG_2$ structure that are calibrated by the 3-form defining the $\GG_2$ structure.  In the case of $\mathcal{S}^7$, 
 associative submanifolds are in fact the links of calibrated 4-dimensional cones in $\R^8$ known as \emph{Cayley} cones.  

Cayley submanifolds of $\R^8$ include complex surfaces and special Lagrangian 4-folds in $\C^4$, 
as well as calibrated submanifolds in $\R^7$ called coassociative 4-folds.
 Examples of associative 3-folds in $\mathcal{S}^7$ are thus provided by well-known classes of submanifolds, 
including \emph{Lagrangian} submanifolds of the nearly K\"ahler 6-sphere, Hopf lifts of \emph{holomorphic curves} in $\C\P^3$ and \emph{minimal Legendrian} submanifolds of $\mathcal{S}^7$.  It is also straightforward to construct associative 3-folds in 
 $\mathcal{S}^7$ using \emph{pseudoholomorphic curves} in the nearly K\"ahler $\mathcal{S}^6$.

\subsection{Motivation and summary}

In this article, we study various aspects of the geometry of associative 3-folds in $\mathcal{S}^7$, motivated by
three natural questions.  

\medskip

Our first question is the following.

\begin{qnA}
Are there examples of associative 3-folds in $\mathcal{S}^7$ which do not arise from other known geometries?
\end{qnA}

The local generality of associative 3-folds in $\mathcal{S}^7$ is greater than the other submanifolds we have listed, so Question A
is only challenging as a global question.  Answering this question is surprisingly difficult -- for example, if one searches for constant 
curvature associative 3-folds, work by the author shows that they are homogeneous and either Lagrangian or minimal 
Legendrian examples which were already well-known.  However, here we provide the first known explicit example of such an associative 3-fold 
in $\mathcal{S}^7$.  We achieve this during the following classification of associative group orbits, which in itself is an important 
problem to solve, building upon the study of Lagrangian group orbits in $\mathcal{S}^6$ in \cite{Mashimo} and 
minimal Legendrian group orbits in $\mathcal{S}^7$ which forms part of the work in \cite{Marshall}.  

\begin{thm}\label{fullthm}
Let $A$ be a connected associative 3-fold in $\mathcal{S}^7$ which is the orbit of a closed 3-dimensional Lie subgroup of\/ 
$\Spin(7)$.  Suppose further that $A$ does not lie in a totally geodesic $\mathcal{S}^6$.  Then, up to rigid motion, $A$ is either
\begin{itemize}
\item[\emph{(a)}] $A_1\cong \U(1)^3$ given by Example \ref{u1ex}, 
\item[\emph{(b)}] $A_2
\cong\SU(2)/\Z_3$ 
given by Example
 \ref{su2ex1}, or 
\item[\emph{(c)}] $A_3\cong\SU(2)$ given by Example \ref{su2ex2}.
\end{itemize}  
\end{thm}

\begin{note}
The example $A_3$ does not arise from other known geometries, thus providing a concrete positive answer to Question A.  The geometry of the associative 3-fold $A_3$ has certainly be identified earlier, at least implicitly, from other viewpoints.  In particular, $A_3$ can be realised as a solution to a system of first order ordinary differential equations by applying the construction methods of \cite{Lotaysym}. 
\end{note}

\noindent Since the subgroup of $\Spin(7)$ fixing a real direction in $\R^8$ is isomorphic to $\GG_2$, if $A$ is a homogeneous 
associative 3-fold lying in a totally geodesic $\mathcal{S}^6$, then it arises as the orbit of 
 a subgroup of $\GG_2$.  Since such associative 3-folds are Lagrangian (or \emph{totally real}) 
submanifolds of $\mathcal{S}^6$, they are classified by  \cite[Theorem 4.4]{Mashimo}.

\medskip

Our second question is motivated by work on minimal submanifolds.

\begin{qnB}
Is there rigidity for associative 3-folds in $\mathcal{S}^7$ satisfying certain curvature conditions; that is, 
is there a classification up to rigid motions?
\end{qnB}

Work by the author shows that there is a unique (up to rigid motion) associative 3-fold in $\mathcal{S}^7$ with a given constant
 curvature.  Moreover, one can show that under certain pinching constraints on scalar 
or sectional curvature an associative 3-fold has to be totally geodesic.  In this article we focus on associative 3-folds satisfying the curvature condition known as \emph{Chen's equality}, which is 
 defined for submanifolds of manifolds with constant curvature.  
Submanifolds satisfying Chen's equality have been studied by a number 
of authors, in particular in \cite{Chen} and \cite{Dillen2}.
 We show that an associative 3-fold satisfying Chen's equality is 
 \emph{austere}: a notion defined in \cite{HarLaw} because of the relationship between austere submanifolds and special Lagrangian 
 geometry.  
Austere 3-folds, which are defined as 3-dimensional submanifolds whose second fundamental form satisfies a pointwise algebraic
 condition, are studied for example in \cite{Bryantaustere} 
and \cite{Ikawa}.  

The constant curvature and austere 
associative 3-folds are certain to play an important role in classifying the so-called \emph{second order families} of associative
 3-folds in $\mathcal{S}^7$; i.e.~those associative 3-folds whose second fundamental form 
satisfies a natural pointwise condition.   Studies of second order families of special submanifolds were started in \cite{BryantSL}
 and have continued in \cite{Foxthesis}, \cite{IonelSL} and \cite{LotayLag}.

  Associative 3-folds satisfying Chen's equality are also 
examples of \emph{ruled} submanifolds:~submanifolds which are fibered by oriented geodesic circles in $\mathcal{S}^7$.  Ruled 
associative 3-folds $\mathcal{S}^7$ are studied in \cite{Fox}, where it is shown  that a generic ruled associative 3-fold can be constructed from a minimal 
 surface $\Sigma$ in $\mathcal{S}^6$ and a holomorphic curve $\Gamma$ in a $\C\P^1$-bundle over $\Sigma$.  The curve 
$\Gamma$ defines an immersion of $\Sigma$ in the Grassmannian $\Gr_+(2,8)$ of oriented 2-planes in $\R^8$
 as a \emph{pseudoholomorphic curve} with respect to a $\Spin(7)$-invariant almost complex structure. 
 We identify a new distinguished class of pseudoholomorphic
 curves in $\Gr_+(2,8)$, which we call \emph{linear} curves, and  prove the following, which 
answers Question B in the case of Chen's equality.  

\begin{thm}\label{mainChenthm}
Let $A$ be an associative 3-fold in $\mathcal{S}^7$ which satisfies Chen's equality.  Then either
\begin{itemize}
\item[\emph{(a)}] $A$ is the Hopf lift of a holomorphic curve in $\C\P^3$ or
\item[\emph{(b)}] $A$ is constructed from a  
 minimal surface $\Sigma$ in $\mathcal{S}^6$ as in Example \ref{ruledex}, where $\Sigma$ is isotropic and the pseudoholomorphic lift of $\Sigma$ in $\Gr_+(2,8)$ is linear.
\end{itemize}
\end{thm}

\medskip

Our final question is inspired by classical work of Bonnet on minimal surfaces in space forms.

\begin{qnC}
Are there associative 3-folds in $\mathcal{S}^7$ which admit non-congruent isometric deformations?
\end{qnC}

Minimal surfaces in space forms always admit isometric deformations, but for higher dimensions and codimensions in spheres 
 little can be said in general.  Infinitesimal deformations of an associative 3-fold $A$ in $\mathcal{S}^7$ are governed by certain
 eigenspinors for a Dirac operator and thus, if $A$ is compact and suitably generic, one would expect $A$ to be rigid up to 
$\Spin(7)$ motions since the index of a Dirac operator on $A$ is zero.  However, we show that generic ruled associative 3-folds in
 $\mathcal{S}^7$ admit a 2-parameter family of local isometric deformations.  Moreover,  we are able to a provide a positive answer to
 Question C by exhibiting examples of associative 3-folds where the local isometric deformations extend to global ones as follows.  

\begin{thm}\label{isomthm}
Given a non-constant curvature 
minimal $\mathcal{S}^2$ in $\mathcal{S}^6$, 
 one may construct an associative 3-fold in $\mathcal{S}^7$, satisfying Chen's equality, which admits
an $\mathcal{S}^1$-family of non-congruent isometric deformations. 
\end{thm}

Since there are a large number of minimal 2-spheres in $\mathcal{S}^6$, this provides many 
examples of 1-parameter deformation families of isometric associative 3-folds in $\mathcal{S}^7$.  Moreover, 
there is a Weierstrass representation for these harmonic 2-spheres and thus for the resulting isometric associative 3-folds.

\begin{remark}
Theorem \ref{isomthm} is especially surprising given that isometric minimal 2-spheres in $\mathcal{S}^6$ are congruent by \cite[Theorem 5.15]{Barbosa}.
\end{remark}

\section[The G2 structure on S7]{The {\boldmath $\GG_2$} structure on {\boldmath $\mathcal{S}^7$}}\label{G2struct}

The key to defining associative submanifolds of the 7-sphere is to introduce a $\GG_2$ structure on $\mathcal{S}^7$, which 
is induced by the standard $\Spin(7)$ structure on $\R^8$.  
We begin by defining distinguished differential forms on $\R^7$ and $\R^8$.

\begin{dfn}\label{phisdfn}  Let $\R^7$ have coordinates $(x_1,\ldots,x_7)$ and let $\R^8$ have coordinates $(x_0,\ldots,x_7)$.  
For convenience we denote the form $\d x_i\w\d x_j\w\ldots\w\d x_k$ by $\d\bfx_{ij\ldots k}$.

We define a 3-form $\varphi_0$ on $\R^7$ and a 4-form $\Phi_0$ on $\R^8$ by:
\begin{align*}
\varphi_0&=\d\bfx_{123}+\d\bfx_{145}+\d\bfx_{167}+\d\bfx_{246}-\d\bfx_{257}-\d\bfx_{347}-\d\bfx_{356}\intertext{and}
\Phi_0&=\d\bfx_{0123}+\d\bfx_{0145}+\d\bfx_{0167}+\d\bfx_{0246}-\d\bfx_{0257}-\d\bfx_{0347}-\d\bfx_{0356}\nonumber\\
&\quad+\d\bfx_{4567}+\d\bfx_{2367}+\d\bfx_{2345}+\d\bfx_{1357}-\d\bfx_{1346}-\d\bfx_{1256}-\d\bfx_{1247}.
\end{align*}
Notice that $\Phi_0$ is self-dual and that if we decompose $\R^8=\R\oplus\R^7$ then 
\begin{equation}\label{Phispliteq}
\Phi_0=\d x_0\w\varphi_0+*\varphi_0,
\end{equation}
where $*\varphi_0$ is the Hodge dual of $\varphi_0$.
By identifying $\R^8$ with $\C^4$ in the usual way, 
we can decompose $\Phi_0$ as follows:
\begin{equation}\label{Phispliteq2}
\Phi_0=\textstyle\frac{1}{2}\,\omega_0\w\omega_0+\Re\Omega_0,
\end{equation}
where $\omega_0$ and $\Omega_0$ are the standard K\"ahler form and holomorphic volume form on $\C^4$. 
It is also worth observing that if $\H$ are the quaternions and we identify $\R^8$ with $\H^2$ in an appropriate way then, by \cite[Lemma 2.21]{BryantHarvey}, we have that
$$\Phi_0=\textstyle\frac{1}{2}\omega_I^2+\textstyle\frac{1}{2}\omega_J^2-\textstyle\frac{1}{2}\omega_K^2,$$
where $\omega_I,\omega_J,\omega_K$ are the K\"ahler forms associated with the triple of complex structures $(I,J,K)$ given by the standard hyperk\"ahler structure on $\H^2$.
\end{dfn}

We can define the simple 14-dimensional exceptional Lie group $\GG_2$ as the stabilizer of $\varphi_0$ in $\GL(7,\R)$. 
Similarly, the stabilizer of $\Phi_0$ in $\GL(8,\R)$ is 
 $\Spin(7)$. Using $\varphi_0$ we can now define general $\GG_2$ 
structures.

\begin{dfn}\label{g2structuredfn}
We say that a 3-form $\varphi$ on an oriented 7-manifold $X$ is \emph{positive} if, for all $x\in X$, 
 $\varphi|_x=\iota_x^*(\varphi_0)$ for some orientation-preserving isomorphism $\iota_x:T_xX\rightarrow\R^7$.  
A positive 3-form $\varphi$ on $X$ determines a unique metric $g_{\varphi}$ on $X$ such that, if $*\varphi$ is the Hodge dual 
to $\varphi$ 
with respect to $g_{\varphi}$ and $g_{\R^7}$ is the Euclidean metric on $\R^7$, the triple $(\varphi,*\varphi,g_{\varphi})$ is identified with
 $(\varphi_0,*\varphi_0,g_{\R^7})$ at each point.  Hence, we may call a positive 3-form 
a \emph{$\GG_2$ structure} on $X$.
\end{dfn}

Since $\GG_2$ is the automorphism group of the octonions $\O$, it is unsurprising that there is 
a relationship between $\varphi_0$ and the cross product structure on $\Im\O$.  In fact, we may identify $\Im\O$
with $\R^7$ so that for all vectors $x,y,z\in\R^7$,
$$\varphi_0(x,y,z)=g_{\R^7}(x\times y,z).$$

This observation leads us to make the following definition.

\begin{dfn}\label{crossprodsdfn}
Let $X$ be an oriented 7-manifold with a $\GG_2$ structure $\varphi$.  In the notation of Definition \ref{g2structuredfn}, 
we define a \emph{cross product} on $X$ via the equation
$$\varphi(x,y,z)=g_{\varphi}(x\times y,z)$$
for $x,y,z\in C^{\infty}(TX)$.
\end{dfn}

\begin{remark}
This relationship between $\GG_2$ structures and cross products allows us to deduce from \cite[Corollary 3.3]{Gray} that a 7-manifold
 admits a $\GG_2$ structure if and only if it is oriented and spin.
\end{remark}

We now define the $\GG_2$ structure on $\mathcal{S}^7$.

\begin{dfn}\label{S7G2dfn}
Write $\R^8\setminus\{0\}=\R^+\times\mathcal{S}^7$ with $r$ being the
 coordinate on $\R^+$ and $\mathcal{S}^7$ the unit 7-sphere.  Then, since $\Phi_0$ is self-dual, we may define a 3-form on $\mathcal{S}^7$ via the formula
 \begin{equation}\label{Phiphieq}
\Phi_0|_{(r,p)}=r^3\d r\w\varphi|_p+r^4\!*\!\varphi|_p,
\end{equation}
where $*$ here denotes the usual Hodge star on $\mathcal{S}^7$.  Moreover, the fact 
that $\d\Phi_0=0$ implies that $\d\varphi=4\!*\!\varphi$. 
 By comparing \eq{Phispliteq} and \eq{Phiphieq} it is clear 
that $\varphi$ is a positive 3-form on $\mathcal{S}^7$ and $g_\varphi$ given in Definition \ref{g2structuredfn} is the round metric on $\mathcal{S}^7$ with constant curvature 1.  
Thus, $\varphi$ is a $\GG_2$ structure on $\mathcal{S}^7$. 
\end{dfn}

\begin{remarks} A $\GG_2$ structure $\varphi$ on $X$ is called
\emph{nearly parallel} if $\d\varphi=4\lambda*\!\varphi$ for some constant $\lambda\neq 0$.  We then call $(X,\varphi)$ a \emph{nearly 
$\GG_2$ manifold} (or a manifold with \emph{weak holonomy} $\GG_2$ in the sense of \cite{Grayweak}).  Nearly parallel $\GG_2$ structures are discussed in some 
detail in \cite{Nearlypar}.  
\end{remarks}

We may also define a contact structure on $\mathcal{S}^7$ using its relationship with $\C^4$.

\begin{dfn}\label{S7contactdfn}
Consider $\R^+\times\mathcal{S}^7=\R^8\setminus\{0\}\cong\C^4\setminus\{0\}$, with $r$ the coordinate on $\R^+$, and let $\omega_0$ be the usual K\"ahler form on $\C^4$.  Since $\omega_0$ is 
a real 2-form and $\d\omega_0=0$, there exists a 1-form $\gamma$ on $\mathcal{S}^7$ such that
\begin{equation}\label{contacteq}
\omega_0|_{(r,p)}=r\d r\w\gamma|_p +\textstyle\frac{1}{2}\,r^2\d\gamma|_p.
\end{equation}
Since $\omega_0^4\neq 0$ we see that $\gamma\w(\d\gamma)^3\neq 0$, so $\gamma$ defines a \emph{contact structure} 
on $\mathcal{S}^7$.
\end{dfn}

As an aside we observe that, just as $\mathcal{S}^7$ receives a nearly parallel $\GG_2$ structure from its inclusion in $\R^8$, 
$\mathcal{S}^6$ inherits a special $\SU(3)$ structure from its inclusion in $\R^7$; namely, 
a \emph{nearly K\"ahler structure}.  We do not discuss this structure fully as it is unnecessary, but we define the induced \emph{almost 
symplectic} and \emph{almost complex} structures on $\mathcal{S}^6$ and the associated distinguished classes of submanifolds. 

\begin{dfn}\label{S6nearlyKdfn}
Write $\R^7\setminus\{0\}=\R^+\times\mathcal{S}^6$ with $r$ the coordinate on $\R^+$ and $\mathcal{S}^6$ the unit 6-sphere. Since $\d\varphi_0=0$,  there exists a 2-form $\omega$ 
on $\mathcal{S}^6$ such that
\begin{equation*}
\varphi_0|_{(r,p)}=r^2\d r\w\omega|_p+\textstyle\frac{1}{3}\,r^3\d\omega|_p.
\end{equation*}
Since $\varphi_0$ is positive, $\omega$ is non-degenerate but \emph{not} closed.  If $g_{\mathcal{S}^6}$ is the round metric on 
$\mathcal{S}^6$, we can define an almost complex structure $J$ on $\mathcal{S}^6$ via the formula  $g_{\mathcal{S}^6}(Ju,v)=\omega(u,v)$ for 
 tangent vectors $u$ and $v$. 

We say that an oriented surface $\Sigma$ in $\mathcal{S}^6$ is a \emph{pseudoholomorphic curve} if 
$\omega|_\Sigma=\vol_\Sigma$ or, equivalently, if $J(T_\sigma\Sigma)=T_\sigma\Sigma$ for all $\sigma\in\Sigma$.  We say that 
an oriented 3-dimensional submanifold $L$ of $\mathcal{S}^6$ is \emph{Lagrangian} if $\omega|_L\equiv 0$.  
\end{dfn}

\section[Associative submanifolds of S7]{Associative submanifolds of {\boldmath $\mathcal{S}^7$}}\label{assoc}

In this section we define and discuss our primary objects of interest.  Note that in general we consider immersed  
submanifolds.

\begin{dfn}\label{asscoassdfn}
Let $X$ be an oriented 7-manifold with a $\GG_2$ structure $\varphi$.  Using the 
notation of Definition \ref{g2structuredfn}, define a vector-valued 3-form $\chi$ on $X$ via 
$$*\varphi(x,y,z,w)=g_{\varphi}\big(\chi(x,y,z),w\big)$$
for vector fields $x,y,z,w$ on $X$.
\begin{itemize}
\item[(a)] An oriented 3-fold $A$ in $X$ is 
\emph{associative} if $\varphi|_A=\vol_A$.  Equivalently, $A$ is associative if $\chi|_A\equiv 0$ and $\varphi|_A>0$.
\item[(b)] An oriented 4-fold $N$ in $X$ is \emph{coassociative} if $*\varphi|_N=\vol_N$.  
  Equivalently, $N$ is coassociative if $\varphi|_N\equiv 0$ and $*\varphi|_N>0$.
\end{itemize}
\end{dfn}

\begin{remarks}
Since the orthogonal complement of an associative 3-plane in $\R^7$ is a coassociative 4-plane, one can equivalently define an 
oriented 3-fold $A$ in $(X,\varphi)$ 
to be associative if $\varphi$ vanishes on the 
normal bundle of $A$.  This characterisation shows that the apparently overdetermined system of equations given by the vanishing of $\chi$ is in fact determined.  Though this property makes this alternative definition more attractive, it is not standard and we elect not to adopt it here.
\end{remarks}

We make an elementary observation, which holds for any nearly $\GG_2$ manifold.

\begin{lem}
There are no coassociative submanifolds of $\mathcal{S}^7$. 
\end{lem}

\begin{proof}
Suppose $N$ is coassociative in $(\mathcal{S}^7,\varphi)$.   
Then $\varphi|_N\equiv 0$ implies that 
$\d\varphi|_N\equiv 0$.  
However, $\d\varphi=4\!*\!\varphi$, so $*\varphi|_N\equiv 0$, yielding our required contradiction.
\end{proof}

We may define distinguished submanifolds of $\R^8$ as follows.  

\begin{dfn}\label{Cayleydfn}
An oriented 4-fold $N$ in $\R^8$ is \emph{Cayley} if $\Phi_0|_N=\vol_N$.  
\end{dfn}

\noindent We can alternatively characterise Cayley 4-folds as follows.

\begin{dfn}\label{taudfn}
Use the notation of Definition \ref{phisdfn}.  
By \cite[Corollary IV.1.29]{HarLaw}, we may define 
a vector-valued 4-form $\tau$ on $\R^8$ such that, if $\tau_j$ is 
the component of $\tau$ in the direction $\frac{\partial}{\partial x_j}$, then $\tau_0=\Phi_0$ and an oriented 4-fold $N$ in $\R^8$ is Cayley if and only if $\tau_j|_N\equiv 0$ for $j=1,\ldots,7$ (up to a choice of orientation so that $\Phi_0|_N>0$). 
The forms $\tau_j$ correspond to the components of the fourfold cross product on $\O$ discussed in \cite[Appendix IV.B]{HarLaw}.  

If we identify $\R^8$ with $\C^4$ as in Definition \ref{phisdfn} we may explicitly calculate:
\begin{align*}
\tau_1&= \Im\Omega_0;\displaybreak[0]\\
\tau_2+i\tau_3&=i(\d z_1\w\d z_2-\d\bar{z}_3\w\d\bar{z}_4)\w\omega_0;\displaybreak[0]\\
\tau_4+i\tau_5&=i(\d z_1\w \d z_3-\d\bar{z}_4\w\d\bar{z}_2)\w\omega_0;\displaybreak[0]\\
\tau_6+i\tau_7&=i(\d z_1\w\d z_4-\d\bar{z}_2\w\d\bar{z}_3)\w\omega_0.
\end{align*}
In particular, we can confirm that $\omega_0|_N\equiv 0$ and $\Im\Omega_0|_N\equiv 0$ force $\tau_j|_N\equiv 0$
for $j=1,\ldots, 7$; i.e.~if $N$ is \emph{special Lagrangian} in $\C^4$ then it is Cayley in $\R^8$.
\end{dfn}

We now give our first basic result concerning associative submanifolds of $\mathcal{S}^7$.
  By a \emph{cone} in $\R^n$ we mean a dilation-invariant subset $C\subseteq\R^n$ and by its \emph{link} we mean the 
subset of the unit $(n-1)$-sphere in $\R^n$ given by $C\cap\mathcal{S}^{n-1}$.

\begin{lem}\label{Cayleyconelem}
A 4-dimensional cone $N$ in $\R^8$ is Cayley if and only if $A=N\cap\mathcal{S}^7$ is associative in $\mathcal{S}^7$.  
\end{lem}

\begin{proof}  
By Definition \ref{Cayleydfn}, $N$ is a Cayley cone in $\R^8$ if and only if $\Phi_0|_N=\vol_N$. 
By \eq{Phiphieq}, $\Phi_0|_N=\vol_N$ if and only if $\varphi|_A=\vol_A$, where $\varphi$ is the $\GG_2$ structure 
on $\mathcal{S}^7$. The result follows from Definition \ref{asscoassdfn}.
\end{proof}

From Lemma \ref{Cayleyconelem} we may deduce the following.

\begin{cor}\label{minrealanalcor}
Associative submanifolds of $\mathcal{S}^7$ are minimal and real analytic wherever they are non-singular.
\end{cor} 

\begin{proof}
This is immediate from the fact that Cayley 4-folds in $\R^8$ have these properties by \cite[Theorem II.4.2]{HarLaw} and 
\cite[Theorem 12.4.3]{JoyceRiem}.
\end{proof}

Given the real analyticity of associative 3-folds in $\mathcal{S}^7$, we can apply the Cartan--K\"ahler Theorem  
as in \cite[Theorem IV.4.1]{HarLaw} to prove the 
following.  

\begin{prop}\label{extendprop}
Let $S$ be an oriented real analytic surface in $\mathcal{S}^7$.  There locally exists a locally 
unique associative 3-fold in $\mathcal{S}^7$ containing $S$.
\end{prop}

\noindent We deduce that associative 3-folds in $\mathcal{S}^7$ locally depend on 4 functions of 2 variables, in the sense of
 exterior differential systems (EDS); that is, the EDS defining associative 3-folds in $\mathcal{S}^7$ is involutive and
  its last non-zero Cartan character is $s_2=4$.  A good reference for the theory of EDS is \cite{BCG3}.     
Proposition \ref{extendprop} is a simple modification of the result \cite[Lemma 2.1]{RoblesSalur}.

\medskip

Using the relationship between Cayley geometry and other geometries we can give some  examples of 
associative 3-folds in $\mathcal{S}^7$.  We begin with the following.

\begin{prop}\label{assocS6prop} Decompose $\R^8=\R\oplus\R^7$ and recall 
 Definition \ref{S6nearlyKdfn}.  
\begin{itemize}
\item[\emph{(a)}]  
 Let $\Sigma$ be an
 oriented surface in $\mathcal{S}^6$ and let $$A=\{(\cos t,\sigma\sin t)\in\R\oplus\R^7\,:\,\sigma\in\Sigma, t\in(0,\pi)\}\subseteq\mathcal{S}^7.$$  Then $A$ is an associative 3-fold in $\mathcal{S}^7$ 
if and only if $\Sigma$ is a pseudoholomorphic curve in $\mathcal{S}^6$. 
\item[\emph{(b)}] Let $L$ be an oriented 3-dimensional submanifold of $\mathcal{S}^6$ and let $A=\{0\}\times L\subseteq\mathcal{S}^7$.
Then $A$ is an associative 3-fold in $\mathcal{S}^7$ if and only if $L$ is a Lagrangian submanifold of $\mathcal{S}^6$.
\end{itemize}
\end{prop}

\begin{proof}
Let $C$ be a 4-dimensional submanifold in $\R^8=\R\oplus\R^7$.  Then $C=\R\times N$ is Cayley if and only if $N$ is associative in
 $\R^7$ by \eq{Phispliteq} and Definitions \ref{asscoassdfn}(a) and \ref{Cayleydfn}.  Similarly, $C=\{0\}\times N$ is Cayley if and only if $N$ is coassociative in $\R^7$ by \eq{Phispliteq} and Definitions \ref{asscoassdfn}(b) and \ref{Cayleydfn}.  Finally, we note that 
a cone $N$ in $\R^7$ is either associative or coassociative if and only if the link of $N$ in $\mathcal{S}^6$ is either a 
pseudoholomorphic curve or a Lagrangian submanifold respectively by Definitions \ref{S6nearlyKdfn} and \ref{asscoassdfn}.
\end{proof}

\noindent Many examples of pseudoholomorphic curves and Lagrangians in $\mathcal{S}^6$ are known 
 (see, for example, \cite{LotayLag}), 
which then give examples of associative 3-folds in $\mathcal{S}^7$.

\begin{prop}\label{minLegprop}
Recall the contact structure $\gamma$ on $\mathcal{S}^7$ given in 
Definition \ref{S7contactdfn} and that an oriented 3-dimensional submanifold $A\subseteq\mathcal{S}^7$ is called Legendrian if 
$\gamma|_A\equiv 0$.  A minimal Legendrian submanifold of 
$\mathcal{S}^7$ is associative.
\end{prop}

\begin{proof}
Let $A$ be Legendrian in $\mathcal{S}^7$.  By \eq{contacteq}, the cone $CA$ is Lagrangian in $\C^4\cong\R^8$.  
  If $A$ is minimal, then $CA$ is minimal and by \cite[Proposition 2.5]{Haskins} is, in fact, \emph{special Lagrangian with phase $e^{i\theta}$} in $\C^4$ for some constant $\theta$; that is, in the notation of 
Definition \ref{phisdfn}, $\omega_0|_{CA}\equiv 0$ and $\Re(e^{-i\theta}\Omega_0)|_{CA}=\vol_{CA}$.  By making a suitable identification of $\C^4\cong\R^8$ we can ensure that $e^{i\theta}=1$, and so $CA$ is 
Cayley in $\R^8$ by \eq{Phispliteq2} and Definition \ref{Cayleydfn}.  The result follows from Lemma \ref{Cayleyconelem}.
\end{proof}

\noindent From the proof of Proposition \ref{minLegprop}, we have that minimal Legendrian submanifolds of $\mathcal{S}^7$ are the
 links of special Lagrangian cones in $\C^4$.  Explicit examples
of such cones are given in \cite[Examples 8.3.5 \& 8.3.6]{JoyceRiem}.  Moreover, a construction is described in 
\cite{HaskinsKap} and \cite{HaskinsKap2} which yields infinitely many topological types of minimal Legendrian, hence associative,
 submanifolds of $\mathcal{S}^7$. 

We may also connect associative 3-folds in $\mathcal{S}^7$ to complex geometry in the following manner.

\begin{prop}\label{holocurveprop}
Let $\bfu:\Sigma\rightarrow\C\P^3$ be a holomorphic curve.  The Hopf fibration of $\mathcal{S}^7$ over $\C\P^3$ induces a circle 
bundle $\mathcal{C}(\Sigma)$ over $\Sigma$.  Let $\bfx:\mathcal{C}(\Sigma)\rightarrow\mathcal{S}^7$ be such 
that the following diagram commutes:
$$
\xymatrix{
\mathcal{C}(\Sigma)\ar[r]^{\bfx}\ar[d] & 
\mathcal{S}^7 \ar[d]
\\
\Sigma\ar[r]^{\bfu}& \C\P^3.}
$$
Then $\bfx\big(\mathcal{C}(\Sigma)\big)$ is an associative 3-fold in $\mathcal{S}^7$.
\end{prop}

\begin{proof}
A holomorphic curve in $\C\P^3$ is the (complex) link of a complex cone in $\C^4$.  Complex surfaces $S$ in $\C^4$ are Cayley in $\R^8$ 
by \eq{Phispliteq2} and Definition \ref{Cayleydfn} since, in the notation of Definition \ref{phisdfn}, $\frac{1}{2}\,\omega_0\w\omega_0|_S=\vol_S$ and $\Omega_0|_S\equiv 0$.
The result follows from Lemma \ref{Cayleyconelem}.
\end{proof}

\begin{remark}
Proposition \ref{holocurveprop} simply states that the (real) link in $\mathcal{S}^7$ of a complex 2-dimensional cone 
in $\C^4$ is associative.
\end{remark}

The associative 3-folds we have defined using $\mathcal{S}^6$ or complex geometry depend locally on either 4 functions of 1 variable
 or 2 functions of 2 variables, so we would expect that the generic associative 3-fold does not arise from these geometries.

\begin{remarks}
There is some overlap between the examples of associative 3-folds in $\mathcal{S}^7$ given by Propositions 
\ref{assocS6prop}-\ref{holocurveprop}.  In particular, we have the following. 
\begin{itemize}
\item[(a)]  A minimal Legendrian surface in a totally geodesic $\mathcal{S}^5$ in $\mathcal{S}^6$ 
is a pseudoholomorphic curve, and so defines an associative 3-fold by Proposition \ref{assocS6prop}(a) which is
 also minimal Legendrian in $\mathcal{S}^7$. 
\item[(b)] A holomorphic curve in a totally geodesic $\C\P^2$ in $\C\P^3$ defines an associative 3-fold by 
Proposition \ref{holocurveprop} which lies in a totally geodesic $\mathcal{S}^6$ and so is Lagrangian by 
Proposition \ref{assocS6prop}(b).
\end{itemize}
\end{remarks}

One might ask about the relationship between associative geometry and the Hopf fibration 
$\mathcal{S}^3\hookrightarrow\mathcal{S}^7\rightarrow\mathcal{S}^4$, which results from viewing $\mathcal{S}^7$ as the 
unit sphere in $\H^2$ and $\mathcal{S}^4\cong\H\P^1$.   
If $\pi:\mathcal{S}^7\rightarrow\mathcal{S}^4$ is the projection and
 $\pi(A)$ is a point, then $A$ is a totally geodesic $\mathcal{S}^3$.  If $\pi(A)$ is a surface, then it is 
the projection of $A$ under 
the fibration $\mathcal{S}^1\hookrightarrow\mathcal{S}^7\rightarrow\C\P^3\rightarrow\mathcal{S}^4$.  
Thus $A$ must be the Hopf lift of a \emph{horizontal} holomorphic curve in $\C\P^3$.  We can view this horizontal holomorphic curve as a 
``twistor lift'' of the surface $\pi(A)$ in $\mathcal{S}^4$ to $\C\P^3$ (c.f.~\cite{Bryant4sphere}).

\section[The structure equations]{The structure equations}\label{struct}

To study associative submanifolds of $\mathcal{S}^7$ we shall think of the 7-sphere as the homogeneous space 
$\Spin(7)/\GG_2$.  Since we are considering $\mathcal{S}^7$ with its $\GG_2$ structure, we may view $\Spin(7)$ as the $\GG_2$ frame 
bundle over $\mathcal{S}^7$.  Therefore, we shall need the structure equations of $\Spin(7)$.

We begin by recalling the following result \cite[Proposition 1.1]{BryantOct}.

\begin{prop}\label{spin7prop}
Extend the elements of\/ $\Spin(7)\subseteq\End(\O)$ complex linearly so that $\Spin(7)\subseteq\End(\C\otimes_{\R}\O)$.  
Using the standard basis of\/ $\C\otimes_{\R}\O$ to represent $\End_{\C}(\C\otimes_{\R}\O)$ as the $8\times 8$ complex-valued 
matrices, the Lie algebra $\spin7$ of $\Spin(7)$ has the following matrix presentation:
\begin{align*}
\spin7&=\left\{\left(\begin{array}{cccc} i\rho & -\bar{\mathfrak{h}}^{\rm T} & 0 & -\theta^{\rm T}\\
\mathfrak{h} & \kappa & \theta & [\bar{\theta}] \\
0 & -\bar{\theta}^{\rm T} & -i\rho & -\mathfrak{h}^{\rm T} \\
\bar{\theta} & [\theta] & \bar{\mathfrak{h}} & \bar{\kappa}\end{array}\right)\,:\,\begin{array}{l}\mathfrak{h},\theta\in
\MM_{3\times 1}(\C),\\[2pt]
\rho\in\R,\,\kappa\in\MM_{3\times 3}(\C),\\[2pt]
\kappa=-\bar{\kappa}^{\rm T},\,\Tr\kappa=-i\rho
\end{array}\right\}
\end{align*}
where 
\begin{equation}\label{sqbrkteq}
[(x\;y\;z)^{\rm T}]=\left(\begin{array}{ccc} 0 & z & -y \\
-z & 0 & x \\
y & -x & 0\end{array}\right).
\end{equation}
\end{prop}

This presentation of $\spin7$ is not currently conducive to the study of associative 3-folds.  We thus make the following 
substitutions:
\begin{gather}
\mathfrak{h}=\frac{1}{2}\left(\begin{array}{c}
(\omega_1+\alpha_1)+i(\omega_2+\alpha_2)\\
(\beta^5_3+\textstyle\frac{4}{3}\,\eta_6)+i(-\beta^6_3+\textstyle\frac{4}{3}\,\eta_5)\\
(\beta^4_3+\textstyle\frac{4}{3}\,\eta_7)+i(-\beta^7_3+\textstyle\frac{4}{3}\,\eta_4)
\end{array}\right);\label{subseq1}\\[2pt]\displaybreak[0]
\theta=\frac{1}{2}\left(\begin{array}{c} 
(\omega_1-\alpha_1)+i(\omega_2-\alpha_2)\\
(-\beta^5_3+\textstyle\frac{2}{3}\,\eta_6)+i(\beta^6_3+\textstyle\frac{2}{3}\,\eta_5)\\
(-\beta^4_3+\textstyle\frac{2}{3}\,\eta_7)+i(\beta^7_3+\textstyle\frac{2}{3}\,\eta_4)
\end{array}\right);\label{subseq2}\displaybreak[0]
\\[2pt]
\label{subseq3} 
\kappa_{11}=-i\alpha_3;
\quad
\kappa_{22}=\textstyle\frac{i}{2}\,(\alpha_3-\omega_3+\gamma_3);
\quad
\kappa_{33}=\textstyle\frac{i}{2}\,(\alpha_3-\omega_3-\gamma_3);
\\[2pt]
\kappa_{21}=\textstyle\frac{1}{2}\,(\beta^6_1+\beta^5_2)+\textstyle\frac{i}{2}\,(\beta^5_1-\beta^6_2);\label{subseq7}\qquad 
\kappa_{31}=\textstyle\frac{1}{2}\,(\beta^7_1+\beta^4_2)+\textstyle\frac{i}{2}\,(\beta^4_1-\beta^7_2);
\\[2pt]
\kappa_{32}=-\textstyle\frac{1}{2}\,\gamma_1-\textstyle\frac{i}{2}\,\gamma_2;\quad\rho=\omega_3,\label{subseq9}
\end{gather}
for real numbers $\omega_j,\alpha_j,\gamma_j,\eta_a,\beta^a_j$ for $j=1,2,3$ and $a=4,5,6,7$ such that
\begin{align}\label{betasymeq1}
&\beta^4_1+\beta^7_2+\beta^6_3=0, &\beta^5_1+\beta^6_2-\beta^7_3=0,\\
&\beta^6_1-\beta^5_2-\beta^4_3=0, &\beta^7_1-\beta^4_2+\beta^5_3=0.\label{betasymeq2}
\end{align}  We deduce the following.

\begin{prop}\label{spin7prop2} Let the index $j$ range from $1$ to $3$ and the index\/ $a$ range from $4$ to $7$.  
We may write the Lie algebra $\spin7$ of\/ $\Spin(7)\subseteq\GL(8,\R)$
as:
\begin{align*}
\spin7&=\left\{\left(\begin{array}{ccc}0 & -\omega^{\rm T} & -\eta^{\rm T} \\
\omega & [\alpha] & -\beta^{\rm T}-\frac{1}{3}\{\eta\}^{\rm T} \\\
\eta & \beta+\frac{1}{3}\{\eta\} & \frac{1}{2}[\alpha-\omega]_++\frac{1}{2}[\gamma]_-  \end{array}\right)
\,:\, \begin{array}{c}\omega=(\omega_j),\alpha=(\alpha_j),\\\gamma=(\gamma_j)\in\MM_{3\times 1}(\R),\\
\eta=(\eta_a)\in\MM_{4\times 1}(\R),\\ \beta=(\beta^a_j)\in\MM_{4\times 3}(\R),\end{array}\right.\\
&\qquad\qquad\qquad\qquad\qquad\qquad \begin{array}{cc}\beta^4_1+\beta^7_2+\beta^6_3=0, &\beta^5_1+\beta^6_2-\beta^7_3=0,\\
\beta^6_1-\beta^5_2-\beta^4_3=0, &\beta^7_1-\beta^4_2+\beta^5_3=0\end{array} \Bigg\},
\end{align*}
where $[(x\;y\;z)^{\rm T}]$ is defined in \eq{sqbrkteq},
\begin{align}
[(x\;y\;z)^{\rm T}]_{\pm}&=\left(\begin{array}{cccc} 0 & -x & -y & \pm z \\
x & 0 & z & \pm y \\
y & -z & 0 & \mp x \\
\mp z & \mp y & \pm x & 0\end{array}\right)\label{sqbrktpmeq}
\intertext{and}
\{(p\;q\;r\;s)^{\rm T}\}&=\left(\begin{array}{ccc} -q & -r & s \\
p & s & r \\
-s & p & -q \\
r & -q & -p \end{array}\right)\label{curlbrkteq}
\end{align}
\end{prop}
The notation $[\,]_{\pm}$ reflects the splitting of $\mathfrak{so}(4)\cong\Lambda^2(\R^4)^*$ into positive and negative subspaces
 (or, equivalently, self-dual and anti-self-dual 2-forms). The conditions on $\beta$ and the symmetries of $\{\,\}$ in 
\eq{curlbrkteq} are related to the cross product on $\Im\O$ defined by $\varphi_0$ given in Definition \ref{phisdfn}. 

Let $\gg:\Spin(7)\rightarrow\GL(8,\R)$ take $\Spin(7)$ to the identity component of the Lie subgroup of $\GL(8,\R)$ 
with Lie algebra $\spin7$.  Write $\gg=(\bfx\;\bfe\;\bff)$ where, for each $p\in\Spin(7)$, $\bfx(p)\in\MM_{8\times 1}(\R)$, $\bfe(p)=(\bfe_1\;\bfe_2\;\bfe_3)(p)\in
\MM_{8\times 3}(\R)$ and $\bff(p)=(\bff_4\;\bff_5\;\bff_6\;\bff_7)(p)\in\MM_{8\times 4}(\R)$.  The Maurer--Cartan form $\phi=\gg^{-1}\d\gg$ 
takes values in $\spin7$, so it can be written as
\begin{equation}\label{phispin7eq}
\phi=\left(\begin{array}{ccc}0 & -\omega^{\rm T} & -\eta^{\rm T} \\
\omega & [\alpha] & -\beta^{\rm T}-\frac{1}{3}\{\eta\}^{\rm T} \\\
\eta & \beta+\frac{1}{3}\{\eta\} & \frac{1}{2}[\alpha-\omega]_++\frac{1}{2}[\gamma]_-  \end{array}\right)
 \end{equation}
for appropriate matrix-valued 1-forms $\omega$, $\eta$, $\alpha$, $\beta$ and $\gamma$.  Moreover, if $\times$ is the cross product 
on $\mathcal{S}^7$ determined by the $\GG_2$ structure as in Definition \ref{crossprodsdfn}, then the symmetries of $\beta$ given in \eq{betasymeq1}-\eq{betasymeq2} can be 
expressed neatly as
$$\sum_{i=1}^3\bfe_i\times(\bff\beta)_j=0\quad\text{for $j=1,2,3$,}$$
using the obvious notation for components of $\bff\beta$.

An associative 3-fold $A$ in $\mathcal{S}^7$ may be locally lifted to $\Spin(7)$ by 
choosing suitably adapted $\GG_2$ frames locally on $A$.  
It is clear that we may adapt frames such that $\bfx$, $\bfe$ and $\omega$ are 
identified with a point in $A$, an orthonormal frame and orthonormal coframe for $A$ respectively.  We thus recognise $\bff$ as an
 orthonormal frame for the normal space to $A$ in $\mathcal{S}^7$ and see that $\eta$ vanishes on $A$. 

From $\d\gg=\gg\phi$ and the Maurer--Cartan equation $\d\phi+\phi\w\phi=0$, we derive the \emph{first} and 
\emph{second} \emph{structure equations} for the adapted frame bundle of an associative 3-fold.

\begin{prop}\label{structprop1}
Let $\gg=(\bfx\;\bfe\;\bff):\Spin(7)\rightarrow\GL(8,\R)$ as described above.
  Write $\phi=\gg^{-1}\d\gg$ as in \eq{phispin7eq} such that $\phi$ takes values in $\spin7$ and  
recall $[\,]_{\pm}$ given in \eq{sqbrktpmeq}.  
On the adapted frame bundle of an associative 3-fold $A$ 
 in $\mathcal{S}^7$, $\bfx:A\rightarrow\mathcal{S}^7$ and $\{\bfe_1,\bfe_2,\bfe_3,
 \bff_4,\bff_5,\bff_6,\bff_7\}$ is a local oriented orthonormal basis for $TA\oplus NA$, so the first structure equations are:
\begin{align}
\d\bfx &= \bfe\omega;\label{poseq}\\
\d\bfe &= -\bfx\omega^{\rm T}+\bfe[\alpha]+\bff\beta;\label{tgteq}\\
\d\bff &= -\bfe\beta^{\rm T}+\textstyle\frac{1}{2}\bff([\alpha-\omega]_++[\gamma]_-).\label{normeq}
\end{align}
\end{prop}

\begin{prop}\label{structprop2}
Use the notation of Proposition \ref{structprop1}.  
On the adapted frame bundle of an associative 3-fold in $\mathcal{S}^7$, there exists a local tensor of functions $h=h^a_{jk}=h^a_{kj}$, 
for $1\leq j,k\leq 3$ and $4\leq a\leq 7$, 
 such that the second structure equations are:
\begin{align}
\d\omega&=-[\alpha]\w\omega;\label{connectioneq}\\
\beta&=h\omega;\label{secondfunformeq}\\
\d[\alpha]&=-[\alpha]\w[\alpha]+\omega\w\omega^{\rm T}+\beta^{\rm T}\w\beta;\label{Gausseq}\\
\d\beta&=-\beta\w[\alpha]-\textstyle\frac{1}{2}([\alpha-\omega]_++[\gamma]_-)\w\beta;\label{Codazzieq}\\
\textstyle\frac{1}{2}\,\d([\alpha-\omega]_++[\gamma]_-)&=-\textstyle\frac{1}{4}[\alpha-\omega]_+\w[\alpha-\omega]_+-\textstyle\frac{1}{4}[\gamma]_-\w[\gamma]_-+\beta\w\beta^{\rm T}.\label{Riccieq}
\end{align}
\end{prop}

\noindent We see from \eq{tgteq} and \eq{connectioneq} that $[\alpha]$ defines the Levi-Civita connection of an associative 3-fold.  
Moreover, \eq{normeq} shows that $\frac{1}{2}([\alpha-\omega]_++[\gamma]_-)$ defines the induced connection on 
the normal bundle of $A$ in $\mathcal{S}^7$.

\begin{notes}
If we set $\hh=(\bfv_0\;\bfv_1\;\bfv_2\;\bfv_3\;\bar{\bfv}_0\;\bar{\bfv}_1\;\bar{\bfv}_2\;\bar{\bfv}_3)$
 where 
 $$\bfv_0=\textstyle\frac{1}{2}\,(\bfx-i\bfe_3),\quad
\bfv_1=\textstyle\frac{1}{2}\,(\bfe_1-i\bfe_2),\quad \bfv_2=\textstyle\frac{1}{2}\,(\bff_6-i\bff_5), \quad\bfv_3=\textstyle\frac{1}{2}\,(\bff_7-i\bff_4),  $$
then $\hh^{-1}\d\hh=\psi$ takes values in $\spin7$ as given in Proposition \ref{spin7prop} and may be written as
\begin{equation}\label{psieq}
\psi=\left(\begin{array}{cccc} i\rho & -\bar{\mathfrak{h}}^{\rm T} & 0 & -\theta^{\rm T}\\
\mathfrak{h} & \kappa & \theta & [\bar{\theta}] \\
0 & -\bar{\theta}^{\rm T} & -i\rho & -\mathfrak{h}^{\rm T} \\
\bar{\theta} & [\theta] & \bar{\mathfrak{h}} & \bar{\kappa}\end{array}\right),
\end{equation}
where $\mathfrak{h}$, $\theta$, $\kappa$ and $\rho$ are given by 
inverting our substitutions \eq{subseq1}-\eq{subseq9}.  Thus, we can recover the structure equations for $\Spin(7)$ given in
 \cite{BryantOct} from $\d\hh=\hh\psi$ and $\d\psi+\psi\w\psi=0$.  
This form for the structure equations will be invaluable in $\S$\ref{chens}-\ref{isom}.
\end{notes}

\subsection[The second fundamental form]{The second fundamental form}

The equations \eq{tgteq} and \eq{secondfunformeq} show that $\beta$ encodes the \emph{second fundamental form} of $A$, essentially 
given by the tensor of functions $h$.  We discuss this formally.    

\begin{dfn}\label{sffdfn}
Let $A$ be an associative 3-fold in $\mathcal{S}^7$ and use the notation of Propositions \ref{structprop1}-\ref{structprop2}.  
The \emph{second fundamental form} $\II_A\in C^{\infty}(S^2T^*\!A; NA)$ of $A$ can be written locally, using 
summation notation, as
$$\II_A=h^a_{jk}\bff_a\otimes\omega_j\omega_k$$
 for a tensor of functions $h$ such that $h^a_{jk}=h^a_{kj}$, $a=4,5,6,7$, $j,k=1,2,3$.  The tensor $h$ is the same tensor that appears in \eq{secondfunformeq} and satisfies the further symmetry conditions for all $j$:
\begin{align}
\label{hsymeq1}
&h^4_{1j}+h^7_{2j}+h^6_{3j}=0;&& h^5_{1j}+h^6_{2j}-h^7_{3j}=0;\\
\label{hsymeq2}
&h^6_{1j}-h^5_{2j}-h^4_{3j}=0;&& h^7_{1j}-h^4_{2j}+h^5_{3j}=0.
\end{align}
These symmetry conditions follow from \eq{betasymeq1}-\eq{betasymeq2} and are equivalent to:
$$\sum_{i=1}^3\bfe_i\times\II_A(\bfe_i,\bfe_j)=0$$
for $j=1,2,3$, using the cross product on $\mathcal{S}^7$.
\end{dfn}

\begin{remark}
The symmetry conditions \eq{hsymeq1}-\eq{hsymeq2} imply that associative 3-folds in $\mathcal{S}^7$ are minimal.  This confirms our earlier result in Corollary \ref{minrealanalcor}.
\end{remark}

We can interpret \eq{Gausseq} as the \emph{Gauss equation} for an associative 3-fold $A$ in $\mathcal{S}^7$ and 
\eq{Codazzieq} as the \emph{Codazzi equation}; i.e.~\eq{Gausseq} relates the Riemann 
curvature to $\II_A$ and \eq{Codazzieq} imposes conditions on the derivative of $\II_A$.  We can also view \eq{Riccieq} as the 
\emph{Ricci equation} relating the curvature of the normal connection to $\II_A$.   

\subsection[Reductions of the structure equations]{Reductions of the structure equations}

We now observe that an associative 3-fold in $\mathcal{S}^7$ which arises from the geometry of $\mathcal{S}^6$ or $\C^4$
 must have certain symmetries in its structure equations.

\begin{note}
In this subsection, and in the remainder of the paper, we will use the notation of Propositions \ref{structprop1}-\ref{structprop2}.
\end{note}

\begin{ex}{\bf (Products)}\label{prodsex}  Suppose that $A$ is an associative 3-fold in $\mathcal{S}^7$ constructed from a pseudoholomorphic
curve $\Sigma$ in $\mathcal{S}^6$ as in Proposition \ref{assocS6prop}(a).  Then a frame for $T\mathcal{S}^7|_A$ may be chosen such that, 
for some function $\lambda$, 
$$\alpha_1=\lambda\omega_2,\quad\alpha_2=-\lambda\omega_1,\quad\text{and}\quad\beta^4_3=\beta^5_3=\beta^6_3=\beta^7_3=0.$$
The latter condition is equivalent to $\II_A(\bfe_3,.)=0$ and implies, by \eq{betasymeq1}-\eq{betasymeq2}, that $\beta^4_2=\beta^7_1$,
 $\beta^5_2=\beta^6_1$, $\beta^6_2=-\beta^5_1$ and $\beta^7_2=-\beta^4_1$.  Here, $\d\omega_3=0$, so that $\bfe_3$ defines the product 
 $\mathcal{S}^1$ direction orthogonal to $\Sigma$. 
 \end{ex}

\begin{ex}{\bf (Lagrangians)}\label{Lagsex} Suppose that $L$ is a Lagrangian submanifold in a totally geodesic $\mathcal{S}^6$ in $\mathcal{S}^7$.  Then $L$ is associative
by Proposition \ref{assocS6prop}(b) and a frame for $T\mathcal{S}^7|_L$ may be chosen such that 
$$\beta^7_1=\beta^7_2=\beta^7_3=0\quad\text{and}\quad\gamma=\alpha-\omega.$$
 In this case, $\bff_7$ is the direction orthogonal to the totally geodesic $\mathcal{S}^6$ containing $L$, and the equations 
\eq{betasymeq1}-\eq{betasymeq2} are equivalent to the statement that
\begin{equation}\label{betaLeq}
\beta_L=\left(\begin{array}{rrr} -\beta^6_1 & -\beta^6_2 & -\beta^6_3\\
\beta^5_1 & \beta^5_2 & \beta^5_3 \\
\beta^4_1 & \beta^4_2 & \beta^4_3 \end{array}\right)
\end{equation}
is a symmetric trace-free matrix of 1-forms.
\end{ex}

\begin{ex}{\bf (Minimal Legendrians)}\label{minLegsex} Suppose that $L$ is a minimal Legendrian submanifold of $\mathcal{S}^7$.  Then $L$ is 
associative by Proposition \ref{minLegprop} and there is a frame for $T\mathcal{S}^7|_L$ such that
$$\beta^7_1=\beta^7_2=\beta^7_3=0\quad\text{and}\quad\gamma=\alpha+\omega.$$
If we define $\beta_L$ as in \eq{betaLeq}, the conditions \eq{betasymeq1}-\eq{betasymeq2} correspond to $\beta_L$ being 
symmetric and trace-free as in Example \ref{Lagsex}.
\end{ex}

\begin{ex}{\bf (Links of complex cones)}\label{holocurveex}
Let $A$ be an associative 3-fold constructed from a holomorphic curve $\Sigma$ 
in $\C\P^3$ as in Proposition \ref{holocurveprop}.  Then 
there is a choice of frame for $T\mathcal{S}^7|_A$ such that
$$\alpha_1=\omega_1,\quad\alpha_2=\omega_2,\quad\text{and}\quad\beta^4_3=\beta^5_3=\beta^6_3=\beta^7_3=0.$$
As in Example \ref{prodsex}, we see that $\II_A(\bfe_3,.)=0$ and $\beta^4_2=\beta^7_1$,
 $\beta^5_2=\beta^6_1$, $\beta^6_2=-\beta^5_1$ and $\beta^7_2=-\beta^4_1$.  Here, $\bfe_3$ defines the direction of the 
 circle fibres of $A$ over $\Sigma$.
\end{ex}

\section[Homogeneous examples]{Homogeneous examples}\label{orbits}

In this section we classify the associative 3-folds in $\mathcal{S}^7$
which arise as the orbits of closed 3-dimensional Lie subgroups $\GG$
 of $\Spin(7)$.
  This is the analogue of the work in \cite{Mashimo} on 
  homogeneous (Lagrangian) links in $\mathcal{S}^6$ of coassociative cones.
  
The most obvious homogeneous associative 3-fold is a totally geodesic 3-sphere as in the following example.

\begin{ex}\label{simpleex}
Let $A_0\subseteq\mathcal{S}^7$ be given by
$$A_0=\{(x_0,x_1,x_2,x_3,0,0,0,0)\in\R^8\,:\,x_0^2+x_1^2+x_2^2+x_3^2=1\}.$$
Then $A_0$ is an associative 3-sphere in $\mathcal{S}^7$ which is invariant under the action of a $\SU(2)^3/\Z_2$ subgroup of $\Spin(7)$ by \cite[Proposition 12.4.2]{JoyceRiem}.  Moreover, $A_0$ is totally geodesic, so has constant curvature 1.
\end{ex}

By \cite[Proposition 12.4.2]{JoyceRiem} we have the following straightforward result.

\begin{prop}\label{simpleprop}
Up to rigid motion, $A_0$ given in Example \ref{simpleex} is the unique totally geodesic 
associative $\mathcal{S}^3$ in $\mathcal{S}^7$.
\end{prop}

Subgroups of $\Spin(7)$ which fix a real direction in $\R^8$
are isomorphic to $\GG_2$.  Therefore, if $\GG$ acts trivially on an $\R$
factor in $\R^8$ it arises as a subgroup of $\GG_2$, and so associative $\GG$-orbits are Lagrangian in a totally geodesic 
$\mathcal{S}^6$.  These are classified by 
\cite{Mashimo}, so we need only consider the case where $\GG$ acts fully on $\R^8$.

We clearly have a $\U(1)^3$ subgroup of $\Spin(7)$ which acts on $\C^4$, and this is the only $\U(1)^3$ subgroup 
up to conjugation.  
 Suppose now that $\GG$ is a subgroup of $\Spin(7)$ which is isomorphic to $\SU(2)$ or $\SO(3)$.  
We have an irreducible representation over the reals $\rho_i$ of $\SU(2)$ on $\R^i$ for $i=3,\ldots,8$.  Therefore, the possible subgroups $\GG$
 can only have one of the following representations:
\begin{equation*}
\rho_3\oplus\rho_5;\quad \rho_4\oplus\rho_4;\quad\text{and}\quad
\rho_8.
\end{equation*}

Suppose $\GG$ corresponds to the representation $\rho_3\oplus\rho_5$. Since $\GG\subseteq\Spin(7)$ and no $\SU(2)$ subgroup of
 $\Spin(7)$ can preserve a decomposition of $\R^8=\R^3\oplus\R^5$ without fixing a real direction, we have a contradiction to the
 assumption that $\GG$ acts fully on $\R^8$.  The second representation
 $\rho_4\oplus\rho_4$ corresponds to the ``diagonal'' action of $\SU(2)$
 on $\C^2\oplus\C^2\cong\C^4$, and $\rho_8$ corresponds to the induced
 action of $\SU(2)$ on $S^3\C^2\cong\C^4$ from the standard action on $\C^2$.
  
We therefore have the following subgroups to consider.

\medskip

\noindent \textbf{(i)} $\U(1)^3$ acting on $\C^4$ as
\begin{equation}\label{u1actioneq}
\left(\begin{array}{c} z_1\\z_2\\z_3\\z_4 \end{array}\right)\mapsto \left(\begin{array}{c}
e^{i\theta_1}z_1 \\ e^{i\theta_2}z_2\\ e^{i\theta_3}z_3\\ e^{i\theta_4}z_4\end{array}\right) \quad \begin{array}{l}\text{for $\theta_1,\theta_2,\theta_3,\theta_4\in\R$}\\ \text{such that $\theta_1+\theta_2+\theta_3+\theta_4=0$.}\end{array}\end{equation}

\noindent \textbf{(ii)} $\SU(2)$ acting on $\C^4$ as
\begin{equation}\label{su2action1eq}
\left(\begin{array}{c} z_1\\z_2\\z_3\\z_4 \end{array}\right)\mapsto \left(\begin{array}{r} az_1+bz_2 \\ -\bar{b}z_1+\bar{a}z_2 \\ az_3+bz_4 \\ -\bar{b}z_3+\bar{a}z_4\end{array}\right)
\quad\begin{array}{l}\text{for $a,b\in\C$}\\ \text{such that $|a|^2+|b|^2=1$.}\end{array}
\end{equation}

\noindent\textbf{(iii)} $\SU(2)$ acting on $\C^4$ as  
\begin{align}\label{su2action2eq}
\left(\begin{array}{c}z_1\\ z_2\\ z_3\\ z_4\end{array}\right)
&\mapsto \left(\begin{array}{c}a^3z_1+\sqrt{3}a^2bz_2
+\sqrt{3}ab^2z_3+b^3z_4\\ -\sqrt{3}a^2\bar{b}z_1+a(|a|^2-2|b|^2)z_2+b(2|a|^2-|b|^2)z_3+\sqrt{3}\bar{a}b^2z_4\\
\sqrt{3}a\bar{b}^2z_1-\bar{b}(2|a|^2-|b|^2)z_2+\bar{a}(|a|^2-2|b|^2)z_3+\sqrt{3}\bar{a}^2bz_4\\
-\bar{b}^3z_1+\sqrt{3}\bar{a}\bar{b}^2z_2-\sqrt{3}\bar{a}^2\bar{b}z_3+\bar{a}^3z_4 \end{array}\right)
\end{align}
for $a,b\in\C$ such that $|a|^2+|b|^2=1$.

\subsection[U(1)3 orbits]{{\boldmath $\U(1)^3$} orbits}

Let $A$ be an associative 3-fold which is an orbit of the action given in
\eq{u1actioneq} and identify $\R^8\cong\C^4$ as in Definition \ref{phisdfn}.  By considering the Lie algebra associated with the group action \eq{u1actioneq}, it is straightforward to see that the tangent space to the cone $CA$ on $A$ at a point $(z_1,z_2,z_3,z_4)\in A$ is spanned by the vectors:
\begin{align*}X_0&=(z_1,z_2,z_3,z_4),& X_1&=(iz_1,0,0,-iz_4),\\ X_2&=(0,iz_2,0,-iz_4),& X_3&=(0,0,iz_3,-iz_4).\end{align*}
Recall the decomposition \eq{Phispliteq2}.  Clearly 
$\omega_0(X_j,X_k)=0$ for $j,k\in\{1,2,3\}$ and so $\omega_0\w\omega_0|_{CA}\equiv 0$.  Thus $CA$ is Cayley if and only if it is special Lagrangian in $\C^4$ by \eq{Phispliteq2}.
  However, the $\U(1)^3$-invariant special Lagrangian 4-folds in $\C^4$ are
   classified in \cite[$\S$III.3.A]{HarLaw} and there is a unique (up to rigid motion) $\U(1)^3$-invariant, non-planar, 
special Lagrangian cone in $\C^4$ given by
\begin{align*}
\{(z_1,z_2,z_3,z_4)\in\C^4\,&:\,|z_1|=|z_2|=|z_3|=|z_4|, \\
&\qquad\qquad\qquad\qquad\,\Re(z_1z_2z_3z_4)=0,
\,\Im(z_1z_2z_3z_4)>0\}.
\end{align*}
This gives us the following example of an associative 3-fold in $\mathcal{S}^7$.

\begin{ex}\label{u1ex}
Let $A_1\subseteq\mathcal{S}^7$ be given by:
\begin{align*}
A_1=\Big\{\textstyle\frac{1}{2}(e^{i\theta_1},e^{i\theta_2},e^{i\theta_3},e^{i\theta_4})\in\C^4\,&:\,
\theta_1,\theta_2,\theta_3,\theta_4\in\R \nonumber\\
&\quad\text{such that}\;\,\theta_1+\theta_2+\theta_3+\theta_4=\frac{\pi}{2}\Big\}.
\end{align*}
Then $A_1$ is a minimal Legendrian, hence associative, 3-fold in $\mathcal{S}^7$
which is invariant under the action of $\U(1)^3$ given in \eq{u1actioneq}.   Moreover, it is straightforward to see that $A_1$ is a 3-torus with constant curvature $0$.
\end{ex}

We deduce the following result.

\begin{prop}\label{u1prop}
Up to rigid motion, $A_1\cong T^3$ given in Example \ref{u1ex} is the unique, connected, non-totally geodesic, $\U(1)^3$-invariant associative 3-fold
in $\mathcal{S}^7$, where the $\U(1)^3$ action is given in \eq{u1actioneq}. 
\end{prop}

\begin{remark} The orbit $A_1$ is fibred by circles over the $\U(1)^2$-invariant minimal Legendrian 2-torus in 
a totally geodesic $\mathcal{S}^5$ in $\mathcal{S}^7$.  
\end{remark}

\subsection[SU(2) orbits 1]{{\boldmath $\SU(2)$} orbits 1}

Let $A$ be an associative 3-fold in $\mathcal{S}^7$ invariant under the 
``diagonal'' $\SU(2)$ action given in \eq{su2action1eq}, where we identify $\C^4\cong\R^8$.  We may easily calculate the Lie algebra associated with \eq{su2action1eq} and deduce that
the tangent space to $A$ at a point $(z_1,z_2,z_3,z_4)$ is thus spanned by the vectors
\begin{align*}
X_1=(z_2,-z_1,z_4,-z_3),\quad X_2=(iz_2,iz_1,iz_4,iz_3),\quad X_3=(iz_1,-iz_2,iz_3,-iz_4).
\end{align*}
By \cite[Lemmas 5.3 \& 5.6]{Mashimo}, using the $\SU(2)$ action we can ensure that the induced metric $g_A$ on $A$ from the round metric on $\mathcal{S}^7$ is given by
 $\lambda_1\omega_1^2+\lambda_2\omega_2^2+\lambda_3\omega_3^2$, where the $\omega_j$ form an orthonormal coframe for $A$ and
 $\lambda_j=|X_j|^2$.  However, $|X_j|^2=1$ for all $j$, thus $A$ is a constant curvature 1 submanifold of the unit 7-sphere.  We
 deduce the following result.

\begin{prop}\label{su2action1prop}  Any associative orbit in $\mathcal{S}^7$ of the $\SU(2)$ action
 given in \eq{su2action1eq} is totally geodesic.
\end{prop}

\begin{note} The non-singular Cayley 4-folds which are invariant under the 
action given in \eq{su2action1eq} were classified in \cite[Theorem 5.3]{Lotaysym} (and, by other means, in \cite{GuPries}).  One can use the proofs of these results to deduce that Cayley cones invariant under \eq{su2action1eq} are unions of Cayley 4-planes, from which Proposition \ref{su2action1prop} follows.
\end{note}

\subsection[SU(2) orbits 2]{{\boldmath $\SU(2)$} orbits 2}

Let $A$ be an associative orbit of the action given in
\eq{su2action2eq} and identify $\R^8\cong\C^4$.  The following matrices provide a basis for the associated Lie algebra to 
\eq{su2action2eq}:
\begin{gather*}
U_1=\left(\begin{array}{cccc} 0 & \sqrt{3} & 0& 0 \\ -\sqrt{3} & 0& 2& 0 \\ 0& -2 & 0& \sqrt{3} \\ 0& 0& -\sqrt{3}& 0\end{array}\right);\qquad
U_2=\left(\begin{array}{cccc} 0 & \sqrt{3}i & 0& 0 \\ \sqrt{3}i & 0& 2i & 0 \\ 0& 2i & 0& \sqrt{3}i \\ 0& 0& \sqrt{3}i& 0\end{array}\right);\displaybreak[0]\\
U_3=\left(\begin{array}{cccc} 3i & 0& 0& 0 \\ 0& i& 0& 0 \\ 0& 0& -i& 0 \\ 0& 0& 0& -3i\end{array}\right).
\end{gather*}
Notice that $[U_2,U_3]=2U_1$, $[U_3,U_1]=2U_2$ and $[U_1,U_2]=2U_3$.
We see immediately that the tangent space to $A$ at $(z_1,z_2,z_3,z_4)$ is spanned by the vectors:
\begin{align*}
X_1&=(\sqrt{3}z_2,-\sqrt{3}z_1+2z_3,-2z_2+\sqrt{3}z_4,-\sqrt{3}z_3);\\ X_2&=(i\sqrt{3}z_2,i\sqrt{3}z_1+2iz_3,2iz_2+i\sqrt{3}z_4,i\sqrt{3}z_3);\displaybreak[0]\\ X_3&=(3iz_1,iz_2,-iz_3,-3iz_4).\end{align*}
By \cite[Lemmas 5.3 \& 5.6]{Mashimo}, we can use the $\SU(2)$ action to ensure that the induced metric $g_A$ on $A$ is given by $\lambda_1\omega_1^2+\lambda_2\omega_2^2+\lambda_3\omega_3^2$, where $\{\omega_1,\omega_2,\omega_3\}$ is an orthonormal coframe for $A$ and
\begin{align}
\lambda_1&=|X_1|^2=4(|z_2|^2+|z_3|^2)-4\sqrt{3}\Re(z_1\bar{z}_3+z_2\bar{z}_4)+3;\label{lambdaeq1}\\
\lambda_2&=|X_2|^2=4(|z_2|^2+|z_3|^2)+4\sqrt{3}\Re(z_1\bar{z}_3+z_2\bar{z}_4)+3;\label{lambdaeq2}\\
\lambda_3&=|X_3|^2=8(|z_1|^2+|z_4|^2)+1,\label{lambdaeq3}
\end{align}
using the fact that $|z_1|^2+|z_2|^2+|z_3|^2+|z_4|^2=1$.  The orthogonality of $X_1$, $X_2$ and $X_3$ forces
\begin{align}
z_1\bar{z}_2-z_3\bar{z}_4=0\quad\text{and}\quad\label{orthogsu2eq1}
\Im(z_1\bar{z}_3+z_2\bar{z}_4)=0.
\end{align}

Let 
\begin{align}
\zeta_1&=3|z_4|^2+|z_3|^2-|z_2|^2-3|z_1|^2;\label{zeta1eq}\\
\zeta_2&=2\sqrt{3}z_1\bar{z}_2+4z_2\bar{z}_3+2\sqrt{3}z_3\bar{z}_4;\label{zeta2eq}\\
\zeta_3&=36z_1^2z_4^2-12z_2^2z_3^2+16\sqrt{3}z_1z_3^3+16\sqrt{3}z_2^3z_4-72z_1z_2z_3z_4.\label{zeta3eq}
\end{align}
It is straightforward to see, using the equation \eq{Phispliteq2} for $\Phi_0$, that, if $CA$ is the cone on $A$, then 
$$\Phi_0|_{CA}=\frac{|\zeta_1|^2+|\zeta_2|^2+\Re\zeta_3}{\sqrt{\lambda_1\lambda_2\lambda_3}}\vol_{CA}$$
 at the point $(z_1,z_2,z_3,z_4)$.  Thus $CA$ is Cayley if and only if
\begin{equation}\label{volsu2eq}
\sqrt{\lambda_1\lambda_2\lambda_3}=|\zeta_1|^2+|\zeta_2|^2+\Re\zeta_3.
\end{equation}
Recall that the Cayley condition on $CA$ is also given by the vanishing 
of the 4-forms $\tau_j$, $j=1,\ldots,7$, given in Definition \ref{taudfn}.  The vanishing of these forms is equivalent to the following equations:
\begin{align}
\Im\zeta_3&=0;\label{tausu2eq1}\\
8(z_1z_2+\bar{z}_3\bar{z}_4)\zeta_1+4(\bar{z}_2\bar{z}_4-z_1z_3)\zeta_2 + 4\sqrt{3}(z_1^2-\bar{z}_4^2)\bar{\zeta}_2&=0;\label{tausu2eq2}\\
4(\sqrt{3}z_2^2-\sqrt{3}\bar{z}_3^2+z_1z_3-\bar{z}_2\bar{z}_4)\zeta_1\qquad\qquad\qquad\qquad\quad&\nonumber\\
-4\sqrt{3}\Re(z_2z_3
+z_1z_4)\zeta_2+8(z_1z_2+\bar{z}_3\bar{z}_4)\bar{\zeta}_2&=0;\label{tausu2eq3}\\
2(3z_2z_3+5\bar{z}_2\bar{z}_3+6z_1z_4-6\bar{z}_1\bar{z}_4)\zeta_1\qquad\qquad\qquad\qquad&\nonumber\\
+4(\bar{z}_2^2-\sqrt{3}z_2z_4)\zeta_2
-4(\bar{z}_3^2-\sqrt{3}z_1z_3)\bar{\zeta}_2&=0.\label{tausu2eq4}
\end{align}

We now make a quick observation.

\begin{lem}\label{sporbitslem}
Let $A$ be an associative orbit through $(z_1,z_2,z_3,z_4)\in\C^4$ of the $\SU(2)$ action in \eq{su2action2eq}, 
where we identify $\C^4$ with $\R^8$ as in Definition \ref{phisdfn}.  Let $\zeta_1,\zeta_2,\zeta_3$ be given by \eq{zeta1eq}-\eq{zeta3eq}.   
\begin{itemize}
\item[\emph{(a)}] $A$ is minimal Legendrian in $\mathcal{S}^7$ if and only if $\zeta_1=\zeta_2=\Im\zeta_3=0$.
\item[\emph{(b)}] $A$ is the link of a complex cone in $\C^4$ if and only if $\zeta_3=0$.
\end{itemize}
\end{lem}

\begin{proof}
Let $\omega_0$ and $\Omega_0$ be the K\"ahler form and holomorphic volume form on $\C^4$ respectively.  Then $\omega_0|_{CA}\equiv 0$ if and only if $\zeta_1=\zeta_2=0$, and $\Omega_0|_{CA}=0$ if and only if $\zeta_3=0$. The result follows.
\end{proof}

As noted in \cite[Remark 5.4]{Mashimo}, it is possible to permute the $\lambda_j$ using the $\SU(2)$ action but still preserve the orthogonality of the $X_j$.  Thus the triple $(\lambda_1,\lambda_2,\lambda_3)$ is well-defined up to permutation and determines the metric on $A$.  We now have the following lemma.

\begin{lem}\label{lambdalem}
Let $A\subseteq\mathcal{S}^7$ be a connected, non-totally geodesic, associative orbit of the $\SU(2)$ action given in \eq{su2action2eq}. Up to permutation, $(\lambda_1,\lambda_2,\lambda_3)$ given by \eq{lambdaeq1}-\eq{lambdaeq3} is either $(3,3,9)$ or $(7,7,1)$.  Moreover, the associative 3-folds with 
 $(\lambda_1,\lambda_2,\lambda_3) =(3,3,9)$
 are orbits through points $(\cos\theta,0,0,\sin\theta)$ for some $\theta\in[0,\frac{\pi}{4}]$, and if $(\lambda_1,\lambda_2,\lambda_3)=(7,7,1)$ then $A$ is the orbit through $(0,\frac{1}{\sqrt{2}}, \frac{i}{\sqrt{2}},0)$.
\end{lem}

\begin{proof}
Suppose first, for a contradiction, that all of the $\lambda_j$ are distinct.  Since we are free to permute the $\lambda_j$ we may assume that $\lambda_3<\lambda_1<\lambda_2$.  From \eq{lambdaeq1} and \eq{lambdaeq2} we see that $\Re(z_1\bar{z}_3+z_2\bar{z}_4)>0$, 
 which means that at least one member of each of the pairs $(z_1,z_4)$ and $(z_2,z_3)$ is non-zero.  Therefore, by \eq{orthogsu2eq1}, there must exist some real number $\mu$ such that 
\begin{equation*}
z_2=\mu z_4\quad\text{and}\quad z_3=\mu z_1.
\end{equation*}
The fact that $\Re(z_1\bar{z}_3+z_2\bar{z}_4)>0$ means that $\mu>0$.  Furthermore,  $\lambda_3<\lambda_1$ only if $\mu>\sqrt{3}$, using \eq{lambdaeq1} and \eq{lambdaeq3}.  
 We deduce from \eq{tausu2eq2} that $\Im(z_1z_4)=0$.  Therefore the equations \eq{tausu2eq1}-\eq{tausu2eq4} are equivalent to the conditions:
\begin{align*}
\Im(z_1^4+z_4^4)&=0;&\Im(z_1z_4)&=0;\\
z_1^2+\bar{z}_4^2-\mu\sqrt{3}(\bar{z}_1^2+z_4^2)&=0;&
\Im\big((z_1^2+\bar{z}_4^2)z_1\bar{z}_4\big)&=0.
\end{align*}
Since $\mu>\sqrt{3}$ and $(z_1,z_4)\neq(0,0)$, 
we see that there are no solutions to these equations and we have reached our required contradiction.

Suppose now that at least two of the $\lambda_j$ are equal.  Then we may assume that $\lambda_1=\lambda_2$ by our earlier remarks. 
 Therefore from \eq{lambdaeq1}-\eq{lambdaeq2} we find that $\Re(z_1\bar{z}_3+z_2\bar{z}_4)=0$.  Combining this information with 
\eq{orthogsu2eq1} we have that
 $$z_1\bar{z}_2-z_3\bar{z}_4=0\quad\text{and}\quad z_1\bar{z}_3+z_2\bar{z}_4=0.$$
Thus,
$$z_1\bar{z}_2\bar{z}_3z_4=|z_3|^2|z_4|^2=-|z_2|^2|z_4|^2=0$$ 
and
$$\bar{z}_1z_2z_3\bar{z}_4=|z_1|^2|z_2|^2=-|z_1|^2|z_3|^2=0.$$
We deduce that either $z_2=z_3=0$ or $z_1=z_4=0$.  Thus the only possible triples $(\lambda_1,\lambda_2,\lambda_3)$ are $(3,3,9)$ and $(7,7,1)$.

For the first case, $z_2=z_3=0$ and $\sqrt{\lambda_1\lambda_2\lambda_3}=9$, so the Cayley condition \eq{volsu2eq} becomes:
$$9 = 9(|z_1|^2-|z_4|^2)^2+36\Re(z_1^2z_4^2).$$
Since $|z_1|^2+|z_4|^2=1$, we see that 
 \eq{volsu2eq} is equivalent to $\Im(z_1z_4)=0$.  Using the 
 action \eq{su2action2eq} we can ensure that both $z_1$ and $z_4$ are real and non-negative, and that $z_1\geq z_4$.  Since $|z_1|^2+|z_4|^2=1$, the result for $(\lambda_1,\lambda_2,\lambda_3)=(3,3,9)$ follows. 

For the second case, $z_1=z_4=0$ and $\sqrt{\lambda_1\lambda_2\lambda_3}=7$, so  we see from \eq{volsu2eq} and \eq{tausu2eq1} that
$$7=(|z_2|^2-|z_3|^2)^2+16|z_2|^2|z_3|^2-12\Re(z_2^2z_3^2)\quad\text{and}\quad \Im(z_2^2z_3^2)=0.$$
Therefore $\Re(z_2z_3)=0$ and $\Im(z_2z_3)=\pm\frac{1}{2}$.  Since $|z_2|^2+|z_3|^2=1$ we deduce that $|z_2|=|z_3|=\frac{1}{\sqrt{2}}$ and $z_3=\pm i\bar{z}_2$.  The result follows by using the $\SU(2)$ action in \eq{su2action2eq}.
\end{proof}

From this lemma we have the following examples of associative 3-folds in $\mathcal{S}^7$.

\begin{ex}\label{su2ex1}
Let $A_{2}(\theta)\subseteq\mathcal{S}^7$ be the orbit of the $\SU(2)$ action \eq{su2action2eq} through $z_\theta=(\cos\theta,0,0,\sin\theta)$, for  
$\theta\in[0,\frac{\pi}{4}]$.  Then $A_2({\theta})$ is associative and has an orthonormal coframe 
$\{\omega_1,\omega_2,\omega_3\}$ such that the induced metric on $A_2(\theta)$ is $3\omega_1^2+3\omega_2^2+9\omega_3^2$.   One sees from 
the action \eq{su2action2eq} that $z_\theta$ has $\Z_3$-stabilizer in $\SU(2)$.  Therefore 
$A_2({\theta})\cong\SU(2)/\Z_3$.  

Furthermore, by Lemma \ref{sporbitslem}, 
$A_2(\theta)$ is the link of a complex cone if and only if $\theta=0$ and it is minimal Legendrian if and only if $\theta=\frac{\pi}{4}$. However, rotation in the $(z_1,z_4)$-plane commutes with the $\SU(2)$ action in \eq{su2action2eq}, so all the $A_2(\theta)$ are congruent up to rigid motion to $A_2=A_2(0)$ which is $\U(2)$-invariant.  
\end{ex}

\begin{remarks}
The fact that $A_2(0)$ and $A_2(\frac{\pi}{4})$ are congruent up to rigid motion is a special case of the main result in \cite{Blair}.
 The minimal Legendrian $A_2(\frac{\pi}{4})$ was first found in \cite[$\S$3.4]{Marshall}.
Clearly $A_2(0)$ is the Hopf lift of the Veronese curve in $\C\P^3$, which is the constant curvature, degree 3 $\C\P^1$ in $\C\P^3$ given explicitly by
$$\{(x^3,\sqrt{3}x^2y,\sqrt{3}xy^2,y^3)\in\C\P^3\,:\,(x,y)\in\C\P^1\}.$$
 The minimal Legendrian $A_2(\frac{\pi}{4})$ is fibred by oriented geodesic circles over the Bor\r{u}vka sphere 
  in $\mathcal{S}^4$, which is a 2-sphere with constant curvature $\frac{1}{3}$. 
\end{remarks}

\begin{ex}\label{su2ex2}
Let $A_3\subseteq\mathcal{S}^7$ be the orbit of the $\SU(2)$ action \eq{su2action2eq} through
 $z=(0,\frac{1}{\sqrt{2}},\frac{i}{\sqrt{2}},0)$.  Then $A_3$ is associative and has an orthonormal coframe $\{\omega_1,\omega_2,\omega_3\}$ such that
 the metric on $A_{3}$ is $7\omega_1^2+7\omega_2^2+\omega_3^2$.  Since the point $z$ has trivial stabilizer under the action 
\eq{su2action2eq}, $A_3\cong\SU(2)$ and the induced metric on $\SU(2)$ is a Berger metric.

Furthermore, by Lemma \ref{sporbitslem} and the classification of homogeneous  links of complex, coassociative and special Lagrangian cones, we see that $A_3$ does not arise from the complex, Lagrangian or minimal Legendrian geometries.
\end{ex}

We now conclude with the following result.

\begin{prop}\label{su2action2prop}
Let $A$ be a connected, non-totally geodesic, associative 3-fold in $\mathcal{S}^7$ which is an orbit of the $\SU(2)$ action given in \eq{su2action2eq}.  Up to rigid motion, either $A=A_2$ 
  as given in Example \ref{su2ex1}, or $A=A_3$ as given in Example \ref{su2ex2}. 
\end{prop}

\noindent Proposition \ref{su2action2prop} 
gives an explicit example of an associative 3-fold in $\mathcal{S}^7$ 
which does not arise from other geometries and hence a new example of a Cayley cone.  

\begin{note}
Combining Propositions \ref{u1prop}, \ref{su2action1prop} and \ref{su2action2prop} leads to Theorem \ref{fullthm}.
\end{note}

\section[Ruled associative 3-folds]{Ruled associative 3-folds}\label{ruledsection}

In this section we review the material we require from \cite{Fox}, which provides the classification of associative 3-folds in
 $\mathcal{S}^7$ that are \emph{ruled} by oriented geodesic circles.  The first key idea will be the relationship 
between ruled associative 3-folds and surfaces in the space of oriented geodesic circles in $\mathcal{S}^7$ which are 
\emph{pseudoholomorphic curves} with respect to a certain almost complex structure.  The other major point is 
that there is a natural connection between these pseudoholomorphic curves and \emph{minimal surfaces} in $\mathcal{S}^6$.

We begin by defining our objects of interest.

\begin{dfn}\label{ruleddfn}
We say that a 3-dimensional submanifold $A$ of $\mathcal{S}^7$ is \emph{ruled} if there exists a smooth fibration $\pi:A
\rightarrow\Gamma$ for some 2-manifold $\Gamma$ whose fibres are oriented geodesic circles in $\mathcal{S}^7$.  
\end{dfn}

\begin{remark}
The associative 3-folds given by Propositions \ref{assocS6prop}(a) and \ref{holocurveprop} provide clear examples 
of ruled associative 3-folds.
\end{remark}

\begin{dfn}\label{Cdfn}
Let $\mathcal{C}$ denote the space of oriented geodesic circles in $\mathcal{S}^7$.  Notice that $\mathcal{C}$ is naturally
 isomorphic to the Grassmannian of oriented 2-planes in $\R^8$.

By \cite[Proposition 2.1]{Fox},  
$\mathcal{C}\cong\Spin(7)/\U(3)$ and the $\U(3)$ structure defines a $\Spin(7)$-invariant (non-integrable) almost 
complex structure $J_{\mathcal{C}}$ on $\mathcal{C}$.  
We call a surface $\Gamma$ in $\mathcal{C}$ a 
\emph{pseudoholomorphic curve} if $J_{\mathcal{C}}(T_p\Gamma)=T_p\Gamma$ for all $p\in\Gamma$.

We can also view $\mathcal{C}$ as the standard 6-quadric in $\C\P^7$ by sending an oriented basis $(\mathbf{a},\mathbf{b})$ for a
 geodesic circle in $\mathcal{S}^7$ to $[\frac{1}{2}(\mathbf{a}-i\mathbf{b})]\in\C\P^7$.  This identification induces the usual 
integrable complex structure $I_{\mathcal{C}}$ on $\mathcal{C}$. 
\end{dfn}

Let ${\bf c}:\Gamma\rightarrow\mathcal{C}$ be a surface in $\mathcal{C}$.  Then we may write
\begin{equation}\label{ceq1}
{\bf c}(p)=\big({\bf a}(p),{\bf b}(p)\big)
\end{equation}
where ${\bf a},{\bf b}:\Gamma\rightarrow\mathcal{S}^7$ are everywhere orthogonal, and so define an oriented basis for a 2-plane 
in $\R^8$ for each $p\in\Gamma$.  We may therefore associate to $\Gamma$ a map 
$\bfx:\Gamma\times[0,2\pi)\rightarrow\mathcal{S}^7$ given by
\begin{equation}\label{ceq2}
\bfx(p,q)={\bf a}(p)\cos q+{\bf b}(p)\sin q.
\end{equation}
Clearly, the image of $\bfx$ is a ruled 3-dimensional submanifold of $\mathcal{S}^7$.  
We can now characterise the associative condition for the image of $\bfx$ in terms of the geometry of $\mathcal{C}$ by 
\cite[Proposition 2.1]{Fox}.

\begin{prop}\label{ruledprop}  Use the notation of Definition \ref{Cdfn}.
\begin{itemize}
\item[\emph{(a)}] A pseudoholomorphic curve ${\bf c}:\Gamma\rightarrow\mathcal{C}$ as given by \eq{ceq1} defines an associative 3-fold in $\mathcal{S}^7$ via $\bfx:\Gamma\times[0,2\pi)\rightarrow\mathcal{S}^7$ as in \eq{ceq2}.
\item[\emph{(b)}] Let $A$ be a ruled associative 3-fold in $\mathcal{S}^7$ as in Definition \ref{ruleddfn}.  Then 
${\mathbf c}:\Gamma\rightarrow\mathcal{C}$ given by ${\bf c}(p)=\pi^{-1}(p)$ is a pseudoholomorphic curve.
\end{itemize}
\end{prop}

 We conclude from Proposition \ref{ruledprop} 
that ruled associative 3-folds are in one-to-one correspondence with 
pseudoholomorphic curves in $\mathcal{C}$.  Hence, they depend locally on 10 functions of 1 variable.  

The next key observation is that we have a \emph{fibration} of $\mathcal{C}$ over the 6-sphere.

\begin{dfn}\label{fibrationdfn}
Let $\mathcal{C}$ be as in Definition \ref{Cdfn} and let ${\bf c}\in \mathcal{C}$.  Since $\mathcal{C}\cong\Spin(7)/\U(3)$, 
${\bf c}$ may be identified with a coset $\gg\U(3)$ for some $\gg\in\Spin(7)$, where $\U(3)$ is the stabilizer of ${\bf c}$ in 
$\Spin(7)$.  As there is a unique $\SU(4)\subseteq\Spin(7)$ containing the $\U(3)$ stabilizer, ${\bf c}$ defines a coset 
$\gg\SU(4)$ in $\Spin(7)/\SU(4)$.  However, $\SU(4)\cong\Spin(6)$, so $\Spin(7)/\SU(4)\cong\mathcal{S}^6$.  Thus, we have a 
$\C\P^3$ fibration $\tau:\mathcal{C}\rightarrow\mathcal{S}^6$.  Explicitly, if ${\bf c}$ has $({\bf a},{\bf b})$ as an 
oriented orthonormal basis, 
then $\tau({\bf c})=
{\bf a}\times {\bf b}$, where $\times$ is the cross product on $\O\cong\R^8$.
\end{dfn}

\noindent Definition \ref{fibrationdfn} describes a ``twistor fibration'' of $\mathcal{C}$ over $\mathcal{S}^6$ in the sense of 
\cite{Salamon2}.

The fibration $\tau:\mathcal{C}\rightarrow\mathcal{S}^6$ allows us to relate pseudoholomorphic curves in 
$\mathcal{C}$ to more familiar surfaces by \cite[Theorem 3.5]{Salamon2}. 

\begin{prop}\label{minprop}
Use the notation of Definitions \ref{Cdfn} and \ref{fibrationdfn}.  Let $\Gamma$ be a pseudoholomorphic curve in $\mathcal{C}$.
\begin{itemize}
\item[\emph{(a)}] If $\tau(\Gamma)=p\in\mathcal{S}^6$, then $\Gamma$ is a holomorphic curve in $\tau^{-1}(p)\cong\C\P^3$.
\item[\emph{(b)}] If $\Gamma$ does not lie in any fibre of $\tau$, then
 $\tau(\Gamma)$ is a  minimal surface in $\mathcal{S}^6$.
\end{itemize}
\end{prop}

\begin{remark}
The minimal surface $\tau(\Gamma)$ will have a finite number of branch points.
\end{remark}
 
Minimal surfaces in $\mathcal{S}^6$ depend locally on 8 functions of 1 variable, whereas 
pseudoholomorphic curves in $\mathcal{C}$ depend on 10 functions of 1 variable locally.  Therefore, we would expect every minimal 
surface in $\mathcal{S}^6$ to be the projection of a pseudoholomorphic curve in $\mathcal{C}$.   Moreover, a general minimal surface should 
lift to a family of pseudoholomorphic curves in $\mathcal{C}$ parameterised by 2 functions of 1 variable locally.  
By \cite[Lemma 7.5]{Fox} one has even more: namely, 
one can construct a family of pseudoholomorphic curves in $\mathcal{C}$ from a minimal surface in $\mathcal{S}^6$ and 
\emph{holomorphic data}.  We now briefly describe the construction.

\begin{ex}\textbf{(Ruled examples)}\label{ruledex}
Let $\bfu:\Sigma\rightarrow\mathcal{S}^6$ be a minimal surface.  
The normal bundle $N\Sigma$ of $\Sigma$ in $\mathcal{S}^6$ has a
 spin structure, so let $\Spin(N\Sigma)$ be the principal spin bundle given by the spin structure.  The associated spinor bundle 
$\mathbb{S}=\Spin(N\Sigma)\times_{\Spin(4)}V$ over $\Sigma$ (where $V$ is the spin representation of $\Spin(4)$) decomposes into positive and 
negative spinor bundles, which we denote by $\mathbb{S}^+$ and $\mathbb{S}^-$.  Since $\Sigma$ has a conformal
 structure, and the positive and negative spinors are interchanged by changing the orientation on $\Sigma$, we may focus on the positive spinors and define the bundle $W^+=\mathbb{S}^+\otimes T^{(1,0)}\Sigma$ over $\Sigma$.  
 Let $\mathcal{X}(\Sigma)
 =\P(W^+)$, which is a holomorphic $\C\P^1$ bundle over $\Sigma$.    
 
 Recall the notation of Definitions \ref{Cdfn} and \ref{fibrationdfn}.
 It is observed in \cite[$\S$7]{Fox} that $\mathcal{X}(\Sigma)$ is contained in $\bfu^*(\mathcal{C})$.  Thus a surface $\Gamma$ in $\mathcal{X}(\Sigma)$ defines a lift of $\Sigma$ to $\mathcal{C}$; i.e.~a smooth map $c_\Gamma:\Sigma\rightarrow\mathcal{C}$ such that $\tau\circ c_{\Gamma}=\bfu$.  By \cite[Lemma 7.5]{Fox}, 
$\Gamma$ is a holomorphic curve in $\mathcal{X}(\Sigma)$ if and only if the lift $c_\Gamma:\Sigma\rightarrow\mathcal{C}$ is a pseudoholomorphic curve.  
This pseudoholomorphic curve defines a ruled associative 3-fold 
in $\mathcal{S}^7$ by Proposition \ref{ruledprop}(a), which we denote by $A(\Sigma,\bfu,\Gamma)$.
\end{ex}

Thus, given a minimal surface $\Sigma$ in $\mathcal{S}^6$, we have a family of ruled associative submanifolds of $\mathcal{S}^7$ locally parameterised by a 
holomorphic function from $\Sigma$ to $\C\P^1$, as we expected from our earlier parameter count.  

By \cite[Theorem 7.8]{Fox}, we may classify the ruled associative 3-folds in $\mathcal{S}^7$.

\begin{thm}\label{ruledthm}
Let $A$ be a ruled associative 3-fold in $\mathcal{S}^7$ as in Definition \ref{ruleddfn}.  Then either
\begin{itemize}
\item[\emph{(a)}] $A$ is given by Proposition \ref{holocurveprop}, or
\item[\emph{(b)}]  $A=A(\Sigma,\bfu,\Gamma)$ 
for some minimal surface $\bfu:\Sigma\rightarrow\mathcal{S}^6$ and holomorphic curve $\Gamma$ in $\mathcal{X}(\Sigma)$ 
as in Example \ref{ruledex}.
\end{itemize}
\end{thm}

\begin{note}
Theorem \ref{ruledthm} is the natural extension of the main result in \cite{LotayLag} for ruled Lagrangian submanifolds of $\mathcal{S}^6$.
\end{note}

To conclude this section we make some further observations concerning ruled associative 3-folds in $\mathcal{S}^7$.

\begin{prop}\label{varpiprop}
 By Proposition \ref{ruledprop} there is an isomorphism $\imath$ between the set of ruled associative 3-folds in $\mathcal{S}^7$ 
and the set of pseudoholomorphic curves in $\mathcal{C}$, as given in Definition \ref{Cdfn}.  Let $\varpi=\tau\circ\imath$, 
where $\tau$ is given in Definition \ref{fibrationdfn}, and let $A$ be a ruled associative 3-fold in $\mathcal{S}^7$. 
\begin{itemize}
\item[\emph{(a)}] $\varpi(A)$ is a point if $A$ is given by Proposition \ref{holocurveprop}.
\item[\emph{(b)}] $\varpi(A)$ is a minimal surface in a totally geodesic $\mathcal{S}^5$ in $\mathcal{S}^6$ if 
$A$ is minimal Legendrian in $\mathcal{S}^7$.
\item[\emph{(c)}] $\varpi(A)$ is a pseudoholomorphic curve or a point in $\mathcal{S}^6$ if 
 $A$ is given by Proposition \ref{assocS6prop}.
\end{itemize}
\end{prop}

\begin{proof}
Recall the notation of Propositions \ref{structprop1}-\ref{structprop2}, giving $\bfx:A\rightarrow\mathcal{S}^7$, $\bfe$ and $\bff$
 which define a moving $\GG_2$ frame for $T\mathcal{S}^7|_A$, and matrices of 1-forms $\alpha$, $\beta$, $\gamma$ 
on the adapted frame bundle of $A$. 
Further, let $\bfe_3$ define the ruling direction.   
If $\times$ is the cross product on $\mathcal{S}^7$ 
given in Definition \ref{crossprodsdfn} we let
\begin{gather*}
\bfu=\bfx\times\bfe_3,\quad\bft_1=\bfx\times\bfe_1,\quad\bft_2=\bfx\times\bfe_2,\\
\bfn_1=\bfx\times\bff_6,\quad\bfn_2=\bfx\times\bff_5,\quad\bfb_1=\bfx\times\bff_7,\quad\bfb_2=\bfx\times\bff_4.
\end{gather*}
From Definition \ref{fibrationdfn}, we see that $\bfu$ defines the immersion of 
$\varpi(A)$ in $\mathcal{S}^6$.  Moreover, using \eq{poseq}-\eq{normeq}, the following structure equations
hold: 
\begin{align}
\d\bfu &=-\bft_1(\alpha_2-\omega_2)+\bft_2(\alpha_1-\omega_1)+\bfn_1\beta^6_3+\bfn_2\beta^5_3+\bfb_1\beta^7_3+\bfb_2\beta^4_3;
\label{surfeq1}\\
\d\bft_1 &= \bfu(\alpha_2-\omega_2)-\bft_2(\alpha_3-\omega_3)+\bfn_1\beta^6_1+\bfn_2\beta^5_1+\bfb_1\beta^7_1+\bfb_2\beta^4_1;
\label{surfeq2}\\
\d\bft_2 &= -\bfu(\alpha_1-\omega_1)+\bft_1(\alpha_3-\omega_3)+\bfn_1\beta^6_2+\bfn_2\beta^5_2+\bfb_1\beta^7_2+\bfb_2\beta^4_2;
\label{surfeq3}\displaybreak[0]\\
\d\bfn_1 &= -\bfu\beta^6_3-\bft_1\beta^6_1-\bft_2\beta^6_2+\textstyle\frac{1}{2}\bfn_2(\alpha_3+\omega_3+\gamma_3)\nonumber\\
&\qquad\qquad+\textstyle\frac{1}{2}\bfb_1(\alpha_1+\omega_1-\gamma_1)-\textstyle\frac{1}{2}\bfb_2(\alpha_2+\omega_2+\gamma_2); 
\label{surfeq4}\\
\d\bfn_2 &= -\bfu\beta^5_3-\bft_1\beta^5_1-\bft_2\beta^5_2-\textstyle\frac{1}{2}\bfn_1(\alpha_3+\omega_3+\gamma_3)\nonumber\\
&\qquad\qquad-\textstyle\frac{1}{2}\bfb_1(\alpha_2+\omega_2-\gamma_2)-\textstyle\frac{1}{2}\bfb_2(\alpha_1+\omega_1+\gamma_1);
\label{surfeq5}\displaybreak[0]\\
\d\bfb_1 &= -\bfu\beta^7_3 -\bft_1\beta^7_1-\bft_2\beta^7_2-\textstyle\frac{1}{2}\bfn_1(\alpha_1+\omega_1-\gamma_1)
+\textstyle\frac{1}{2}\bfn_2(\alpha_2+\omega_2-\gamma_2)\nonumber\\
&\qquad\qquad+\textstyle\frac{1}{2}\bfb_2(\alpha_3+\omega_3-\gamma_3);\label{surfeq6}\displaybreak[0]\\
\d\bfb_2 &= -\bfu\beta^4_3-\bft_1\beta^4_1-\bft_2\beta^4_2+\textstyle\frac{1}{2}\bfn_1(\alpha_2+\omega_2+\gamma_2)+\textstyle\frac{1}{2}\bfn_2(\alpha_1+\omega_1+\gamma_1) \nonumber\\
&\qquad\qquad-\textstyle\frac{1}{2}\bfb_1(\alpha_3+\omega_3-\gamma_3).\label{surfeq7}
\end{align}

We immediately deduce from \eq{surfeq1}-\eq{surfeq7} that the image of $\bfu$ is a point in $\mathcal{S}^6$ 
if and only if $\alpha_1=\omega_1$, $\alpha_2=\omega_2$ 
 and $\beta^4_3=\beta^5_3=\beta^6_3=\beta^7_3=0$.  From Example \ref{holocurveex} we see that these conditions are satisfied by the Hopf lift of 
 a holomorphic curve in $\C\P^3$, from which (a) follows immediately.   

Hence,  \eq{surfeq1}-\eq{surfeq7} are the structure equations for a 
 surface in $\mathcal{S}^6$, where $\bft_1,\bft_2$ are orthogonal unit tangent vectors and $\bfn_1,\bfn_2,\bfb_1,\bfb_2$ are 
orthogonal unit normal vectors.  
Furthermore, $\Omega_1=-\alpha_2+\omega_2$ and $\Omega_2=\alpha_1-\omega_1$ define an orthonormal coframe for $\varpi(A)$ 
and it is straightforward to calculate the second fundamental form of $\varpi(A)$ and verify that $\varpi(A)$ is minimal.  

We observe that $\d\bfb_1=0$, so $\varpi(A)$ is contained in a totally geodesic $\mathcal{S}^5$, if and only if
 $\beta^7_1=\beta^7_2=\beta^7_3=0$ and $\gamma=\alpha+\omega$.   Part (b) follows from the observations in Example \ref{minLegsex}.

For part (c), first notice that if we take a pseudoholomorphic
 curve $\Sigma$ in $\mathcal{S}^6$ and construct an associative 3-fold $A$ from it as in Proposition \ref{assocS6prop}(a), 
then $\varpi(A)\cong\Sigma$.  By Proposition \ref{assocS6prop}(b), an associative 3-fold $A$ lying in a totally geodesic
 $\mathcal{S}^6$ is Lagrangian, so it is ruled over a pseudoholomorphic curve in $\mathcal{S}^6$ or is the Hopf lift of 
a holomorphic curve in $\C\P^2$ by \cite[Theorem 7.5]{LotayLag}.  Thus $\varpi(A)$ is either a pseudoholomorphic curve or a point 
as claimed.
\end{proof}

\begin{remarks}
\begin{itemize}\item[]
\item[(a)] Though $A_1$ and $A_3$ given in Examples \ref{u1ex} and \ref{su2ex2} are fibered by circles over a surface, they are 
not ruled as the circles are not geodesics in $\mathcal{S}^7$.  
\item[(b)] The associative 3-folds $A_2(\theta)$ given in Example \ref{su2ex1} are ruled.  Moreover, $\varpi\big(A_2(0)\big)$ is 
a point and $\varpi\big(A_2(\frac{\pi}{4})\big)$ is the Bor\r{u}vka sphere $\mathcal{S}^2(\frac{1}{3})$ in a
 totally geodesic $\mathcal{S}^4$ in $\mathcal{S}^6$.
\end{itemize}
\end{remarks}

\section[Chen's equality]{Chen's equality}\label{chens}
 
In the study of submanifolds of manifolds with constant curvature, it has been fruitful to analyse those 
submanifolds satisfying the curvature condition known as \emph{Chen's equality}, particularly in the case when the submanifold 
is minimal.  In this section we classify the associative 3-folds in $\mathcal{S}^7$ satisfying Chen's equality.  

We begin by defining the curvature condition we wish to study.

\begin{dfn}\label{Chendfn}
Let $A$ be an associative 3-fold in $\mathcal{S}^7$.  Let $S_A$ be the scalar curvature of $A$ and let $K_A(p)$ be the minimum of the sectional curvatures of $A$ at $p$.  Since $A$ is minimal by 
Corollary \ref{minrealanalcor}, it follows from \cite[Lemma 3.2]{Chen} that
$$\delta_A=\textstyle\frac{1}{2}S_A-K_A\leq 2.$$ 
The curvature condition $\delta_A=2$, introduced in \cite{Chen}, is known as \emph{Chen's equality}.  
By \cite[Lemma 3.2]{Chen}, Chen's equality is equivalent to the existence of a non-zero tangent vector $\bfv$ on 
$A$ at each point such that $\II_A(\bfv,.)=0$, where $\II_A$ is the second fundamental form of $A$.
\end{dfn}

The Lagrangian submanifolds of $\mathcal{S}^6$ which satisfy Chen's equality are classified in \cite{Dillen2}.  
  These submanifolds are necessarily associative when embedded in $\mathcal{S}^7$ by Proposition \ref{assocS6prop}(b).  
 We now give some further examples of associative 3-folds in $\mathcal{S}^7$ which satisfy Chen's equality.

\begin{prop}\label{Chenprop2}
Associative 3-folds given by Propositions \ref{assocS6prop}(a) and 
\ref{holocurveprop} satisfy Chen's equality. 
\end{prop}

\begin{proof}
  The result follows immediately from the observations in Examples \ref{prodsex} and \ref{holocurveex} and 
Definition \ref{Chendfn}.
\end{proof}

The examples of associative 3-folds in $\mathcal{S}^7$ satisfying Chen's equality given by Proposition \ref{Chenprop2} 
depend locally on 4 functions of 1 variable.  However, it is a straightforward calculation using exterior differential systems to 
see that associative 3-folds satisfying Chen's equality depend on 6 functions of 1 variable.  
We begin our analysis of these general examples with the following key result.  

\begin{lem}\label{ruledlem}
An associative 3-fold in $\mathcal{S}^7$ which satisfies Chen's equality is ruled.
\end{lem}

\begin{proof}  Let $\II_A$ be the second fundamental form of an associative 3-fold $A\subseteq\mathcal{S}^7$ satisfying Chen's 
equality and recall the notation of Propositions \ref{structprop1}-\ref{structprop2}.  
By Definition \ref{Chendfn} there exists a local unit tangent vector field $\bfe_3$ on $A$ such that $\II_A(\bfe_3,.)=0$.  
We thus have that $\II_A$ must locally be of the form:
\begin{align}
\II_A&=\bff_4\otimes(s(\omega_1^2-\omega_2^2)+2t\omega_1\omega_2)+\bff_5\otimes(v(\omega_1^2-\omega_2^2)+2u\omega_1\omega_2) 
\nonumber\\
&+\bff_6\otimes(u(\omega_1^2-\omega_2^2)-2v\omega_1\omega_2)+\bff_7\otimes(t(\omega_1^2-\omega_2^2)-2s\omega_1\omega_2).\label{Chensffeq}
\end{align}
From \eq{Codazzieq} we quickly deduce that the connection forms $\alpha_1$ and $\alpha_2$ on $A$ must be linear combinations of 
$\omega_1$ and $\omega_2$.  Thus, it follows from the first structure equations \eq{poseq} and \eq{tgteq} given in 
Proposition \ref{structprop1} that
$$\d\bfx=\bfe_3\omega_3,\quad\d\bfe_1=0,\quad\d\bfe_2=0\quad\text{and}\quad\d\bfe_3=-\bfx\omega_3\quad\text{modulo $\omega_1,\omega_2$}.$$
These equations define circles in $A$ which are geodesics in $\mathcal{S}^7$.  Thus, $A$ is ruled and $\bfe_3$ defines the 
direction of the ruling.
\end{proof}

Recall that the ruled associative 3-folds depend locally on 10 functions of 1 variable.  Thus, associative
3-folds satisfying Chen's equality form a distinguished subfamily of the ruled family.  By Theorem \ref{ruledthm}, to identify this
 subfamily we must understand Chen's equality as a condition on minimal surfaces in $\mathcal{S}^6$ and their pseudoholomorphic lifts
 to the space of oriented geodesic circles in  $\mathcal{S}^7$. 
 
We begin by identifying the distinguished minimal surfaces in $\mathcal{S}^6$ we seek. 

\begin{dfn}\label{ellipsedfn}
Let $\Sigma$ be a surface in $\mathcal{S}^6$ and let $\II_\Sigma$ be its second fundamental 
form.  The \emph{first ellipse of curvature} at $p\in\Sigma$ is given by
$$
\{\II_\Sigma(\bfv,\bfv)\,:\,\bfv\in T_p\Sigma,\,|\bfv|=1\}\subseteq N_p\Sigma.$$
This is an ellipse in the normal space $N_p\Sigma$ whenever it is non-degenerate.  

Following \cite{ONeill}, we say that $\Sigma$ is \emph{isotropic} if its first ellipse of curvature is a (possibly degenerate) 
circle at each point.    
\end{dfn}

\begin{remarks}
  If $\Sigma$ is a surface in a Riemannian manifold $M$ we may define the third fundamental form  $\III_\Sigma\in
C^{\infty}(S^3T^*\Sigma;N\Sigma)$ by setting $\III_\Sigma(X,Y,Z)$ to be the component of $\nabla_X\II_\Sigma(Y,Z)$ which is 
orthogonal to $T\Sigma$ and the image of $\II_\Sigma$, where $\nabla$ is the Levi-Civita connection on $M$.
The second ellipse of curvature at $p\in\Sigma$ is given by
$
\{\III_\Sigma(\bfv,\bfv)\,:\,\bfv\in T_p\Sigma,\,|\bfv|=1\}\subseteq N_p\Sigma$.  

A minimal surface in $\mathcal{S}^6$ whose first and second ellipses of curvature are all circles or points is called \emph{superminimal}.  We should point out that 
 ``isotropic'' is over-used 
and that a number of authors use it to mean ``superminimal''.
 \end{remarks}

\begin{prop}\label{Chenvarpiprop}
Let $A$ be an associative 3-fold in $\mathcal{S}^7$ which satisfies Chen's equality and use the notation of Proposition 
\ref{varpiprop} and Definition \ref{ellipsedfn}.  Then $\varpi(A)\subseteq\mathcal{S}^6$ is either a point or an isotropic 
 minimal surface. 
\end{prop}

\begin{proof}
  Since $A$ satisfies Chen's equality, its second fundamental form can be written locally as in \eq{Chensffeq}.  
  However, we still have freedom to transform the orthonormal frame for $T\mathcal{S}^7|_A$ to set $s=t=v=0$.  Thus, locally,
\begin{equation}\label{Chensffeq2}
\II_A=2u\bff_5\otimes\omega_1\omega_2+u\bff_6\otimes(\omega_1^2-\omega_2^2).
\end{equation}
We need only consider the case where $\varpi(A)=\Sigma$ is a surface by Propositions \ref{varpiprop}(a) and \ref{Chenprop2}. Recall the proof of Proposition \ref{varpiprop}.  
From the Codazzi equation \eq{Codazzieq} and 
the structure equations \eq{surfeq2}-\eq{surfeq3} for $\Sigma$, 
we calculate the second fundamental form of $\Sigma$ to be given locally by:
\begin{equation*}
\II_\Sigma=\bfn_1\otimes\big(U(\Omega_1^2-\Omega_2^2)-2V\Omega_1\Omega_2\big)+\bfn_2\otimes\big(V(\Omega_1^2-\Omega_2^2)
+2U\Omega_1\Omega_2\big),
\end{equation*}
where $\Omega_1=-\alpha_2+\omega_2$, $\Omega_2=\alpha_1-\omega_1$ and $U,V$ are multiples of $u$ determined by $\alpha_1$ and 
$\alpha_2$.  Thus the first ellipse of curvature of $\Sigma$ at a point $p$ is given by:
$$ \big\{\cos q\big(U(p)\bfn_1(p)+V(p)\bfn_2(p)\big)+\sin q\big(U(p)\bfn_2(p)-V(p)\bfn_1(p)\big)\,:\,q\in[0,2\pi)\}.$$
Since this is clearly a circle or a point for each $p$ the result follows.
\end{proof}

\begin{remark}
The associative 3-fold $A_2$ in Example \ref{su2ex1} satisfies Chen's equality.  
\end{remark}

From Propositions \ref{minprop} and \ref{Chenvarpiprop} we see that the pseudoholomorphic curves in $\mathcal{C}$ corresponding to the 
associative 3-folds in $\mathcal{S}^7$ satisfying Chen's equality must either lie in a fibre of the fibration $\tau:\mathcal{C}
\rightarrow\mathcal{S}^6$ given in Definition \ref{fibrationdfn} or
project to an isotropic minimal surface under $\tau$.  
  In fact, we can describe these pseudoholomorphic curves using $\tau$, but first we must study the structure equations 
  for a pseudoholomorphic curve in $\mathcal{C}$.

Given a pseudoholomorphic curve $\mathbf{c}:\Gamma\rightarrow\mathcal{C}$, we can choose orthogonal
 $\mathbf{a},\mathbf{b}:\Gamma\rightarrow\mathcal{S}^7$ such that $\mathbf{c}(p)$ has $\big(\mathbf{a}(p),\mathbf{b}(p)\big)$ as an oriented orthonormal basis for each $p\in\Gamma$.   
 Define $\bfv_0=\frac{1}{2}(\mathbf{a}-i\mathbf{b})$ and let $\bfv=(\bfv_1\;\bfv_2\;\bfv_3)$ be a unitary frame for
 $T^{(1,0)}\mathcal{C}|_\Gamma$ adapted such that $\bfv_1$ spans $T^{(1,0)}\Gamma$, so that $\bfv_2,\bfv_3$ are normal to
 $\Gamma$ in $\mathcal{C}$.  
 Finally, we can form  
$\hh=(\bfv_0\;\bfv_1\;\bfv_2\;\bfv_3\;\bar{\bfv}_0\;\bar{\bfv}_1\;\bar{\bfv}_2\;\bar{\bfv}_3)$.

Since we can view $\Spin(7)$ as a $\U(3)$ bundle over $\mathcal{C}$, 
$\hh^{-1}\d\hh=\psi$ takes values in $\spin7$ as given in Proposition \ref{spin7prop} and may be written as in \eq{psieq},
using 1-forms $\mathfrak{h}$, $\theta$, $\kappa$ and $\rho$ on $\Spin(7)$ taking values in the appropriate spaces of matrices as 
described in Proposition \ref{spin7prop}.   
Thus, $\d\hh=\hh\psi$ and $\d\psi+\psi\w\psi=0$ are the structure equations for the adapted frame bundle over $\Gamma$.
Moreover, the complex-valued 1-forms given by $\mathfrak{h}$ and $\theta$ define the almost complex structure $J_{\mathcal{C}}$ on $\mathcal{C}$ 
by declaring a 1-form on $\mathcal{C}$ to be a $(1,0)$-form if its pullback to $\Spin(7)$ lies in the subbundle 
defined by the $\mathfrak{h}_i$ and $\theta_j$.  From these considerations we can immediately deduce the following result.

\begin{prop}\label{pholostructprop}
Recall Definition \ref{Cdfn} and let $\Gamma$ be a pseudoholomorphic curve in $\mathcal{C}$. 
Using the notation above, the first structure equations for 
the adapted frame bundle over $\Gamma$ are given by:   
\begin{align}
\d\bfv_0&=i\bfv_0\rho+\bfv\mathfrak{h}+\bar{\bfv}\bar{\theta};\label{pholostructeq1}\\
\d\bfv&=-\bfv_0\bar{\mathfrak{h}}^{\rm T}+\bfv\kappa-\bar{\bfv}_0\bar{\theta}^{\rm T}+\bar{\bfv}[\theta],\label{pholostructeq2}
\end{align}
where $[\,]$ is defined in \eq{sqbrkteq}.  
The second structure equations for the adapted frame bundle over $\Gamma$ are given by:
\begin{align}
\d\mathfrak{h}&=-(\kappa-i\rho\Id)\w\mathfrak{h};\label{pholostructeq3}\\
\d\theta&=-(\kappa+i\rho\Id)\w\theta;\label{pholostructeq4}\\
\d\kappa&=-\kappa\w\kappa+\mathfrak{h}\w\bar{\mathfrak{h}}^{\rm T}+\theta\w\bar{\theta}^{\rm T}-[\bar{\theta}]\w[\theta],
\label{pholostructeq5}
\end{align}
where $\Id$ is the $3\times 3$ identity matrix.
\end{prop}

We see from \eq{pholostructeq3}-\eq{pholostructeq4} that the vectors of forms $\mathfrak{h}$ and $\theta$ push down to 
the pseudoholomorphic curve $\Gamma$.
It is observed in \cite[$\S$5]{Fox} that $\mathfrak{h}$ and $\theta$ are related to the fibration $\tau$ given 
in Definition \ref{fibrationdfn} in the following manner.

\begin{dfn}\label{horverdfn}
Use the notation of Definitions \ref{Cdfn} and \ref{fibrationdfn} and Proposition \ref{pholostructprop} and 
consider $\mathcal{C}$ with its almost complex structure $J_{\mathcal{C}}$.  
Let $\mathcal{H}$ and $\mathcal{V}$ be the subbundles of $T^{(1,0)}\mathcal{C}$ spanned by the vectors on which 
  $\mathfrak{h}$ and $\theta$ vanish respectively.  Then $T^{(1,0)}\mathcal{C}=\mathcal{H}\oplus\mathcal{V}$ and, 
 by \cite[Lemma 5.1]{Fox}, $\tau_*:\mathcal{H}\rightarrow T^{(1,0)}\mathcal{S}^6$ is an isomorphism, where we consider $\mathcal{S}^6$ with the almost complex structure described in Definition \ref{S6nearlyKdfn}.  Thus, 
 $\mathcal{H}$ and $\mathcal{V}$ can be viewed as horizontal and vertical distributions with respect to the fibration 
 $\tau:\mathcal{C}\rightarrow\mathcal{S}^6$.
 
A pseudoholomorphic curve $\Gamma$ in $\mathcal{C}$ is \emph{vertical} if $\theta|_\Gamma=0$.  
  By \cite[Lemma 5.1 \& Corollary 5.3]{Fox}, 
$\Gamma$ is vertical if and only if it lies in a single fibre of $\tau$, and may be viewed as a 
  holomorphic curve in $\C\P^3$.

A pseudoholomorphic curve $\Gamma$ in $\mathcal{C}$ is \emph{horizontal} if $\mathfrak{h}|_\Gamma=0$.  
  By \cite[Corollary 5.3]{Fox}, a horizontal pseudoholomorphic curve is a holomorphic 
curve with respect to the complex structure $I_{\mathcal{C}}$ on $\mathcal{C}$, and is algebraic in 
the 6-quadric in $\C\P^7$. 
\end{dfn}

We now define a new family of pseudoholomorphic curves in $\mathcal{C}$.

\begin{dfn}\label{lineardfn}
Recall the notation of Definition \ref{Cdfn} and Proposition \ref{pholostructprop}.  Let 
$\lambda=\mathfrak{h}\times\theta$; i.e.
\begin{equation*}
\lambda_1=\mathfrak{h}_2\circ\theta_3-\mathfrak{h}_3\circ\theta_2,\quad
\lambda_2=\mathfrak{h}_3\circ\theta_1-\mathfrak{h}_1\circ\theta_3,\quad
\lambda_3=\mathfrak{h}_1\circ\theta_2-\mathfrak{h}_2\circ\theta_1.
\end{equation*}
By \eq{pholostructeq3}-\eq{pholostructeq4}, $\d\lambda=-2\kappa\w\lambda$, so $\lambda$ pushes down to pseudoholomorphic curves in $\mathcal{C}$.
We say that a pseudoholomorphic curve $\Gamma$ in $\mathcal{C}$ is \emph{linear} if $\lambda|_\Gamma=0$.  
\end{dfn}

Clearly, horizontal and vertical pseudoholomorphic curves in $\mathcal{C}$ are linear.
Using exterior differential systems, one sees that the linear pseudoholomorphic curves depend locally on 6 functions of 1 variable,
whereas the horizontal and vertical curves depend locally on 4 functions of 1 variable.  

The utility of Definition \ref{lineardfn} becomes apparent in our next result.

\begin{prop}\label{Chenlinearprop}
Recall the notation of Definitions \ref{Cdfn} and \ref{lineardfn}.  
Associative 3-folds in $\mathcal{S}^7$ satisfying Chen's equality are in one-to-one correspondence with 
linear pseudoholomorphic curves in $\mathcal{C}$.
\end{prop}

\begin{proof}
 Recall the notation of Propositions \ref{structprop1}-\ref{structprop2} and \ref{pholostructprop}.  
 
Let $A$ be an associative submanifold of $\mathcal{S}^7$ satisfying Chen's equality.  One can 
adapt frames over $A$ so that the second fundamental form $\II_A$ of $A$ has the local form as in \eq{Chensffeq2}.  
  Since $A$ is ruled by Lemma \ref{ruledlem}, $A$ defines a pseudoholomorphic curve $\Gamma$ in $\mathcal{C}$ as in 
Proposition \ref{ruledprop}(b).  
From \eq{subseq1}-\eq{subseq2} and \eq{Chensffeq2} we see that the vectors of 1-forms 
$\mathfrak{h}$ and $\theta$ appearing in the structure equations \eq{pholostructeq1}-\eq{pholostructeq5}
in Proposition \ref{pholostructprop} satisfy $\mathfrak{h}=(\mathfrak{h}_1,0,0)$ and $\theta=(\theta_1,0,0)$.  
Thus $\Gamma$ is linear.

Suppose now that $\Gamma$ is a linear pseudoholomorphic curve in $\mathcal{C}$ and let $A$ be the ruled associative 
3-fold constructed as in Proposition \ref{ruledprop}(a).  We see from \eq{subseq1}-\eq{subseq2}, since $\eta|_A=0$, that we have the freedom to adapt frames on $\Gamma$ such that
 $\mathfrak{h}_2+\theta_2=\mathfrak{h}_3+\theta_3=0$.  
Thus the condition $\lambda|_\Gamma=0$ implies that 
\begin{align*}
\mathfrak{h}_j\circ\theta_1-\theta_j\circ\mathfrak{h}_1=\mathfrak{h}_j\circ(\mathfrak{h}_1+\theta_1)=0
\end{align*}
on $\Gamma$ for $j=2,3$.  From \eq{subseq1}-\eq{subseq2} we see that $\mathfrak{h}_1+\theta_1=\omega_1+i\omega_2$ never vanishes, so 
$\mathfrak{h}_2=\theta_2=\mathfrak{h}_3=\theta_3=0$.  We deduce that $A$ satisfies Chen's equality from \eq{subseq1}-\eq{subseq2} and
 Definition \ref{Chendfn}.
\end{proof}

From Theorem \ref{ruledthm} and Propositions \ref{Chenvarpiprop} and \ref{Chenlinearprop} we have the following. 

\begin{cor}\label{Chenlinearcor}
Use the notation of Definitions \ref{Cdfn}, \ref{fibrationdfn}, \ref{ellipsedfn} and \ref{lineardfn}.
Every isotropic minimal surface $\bfu:\Sigma\rightarrow\mathcal{S}^6$ defines a linear pseudoholomorphic curve 
 $\mathbf{c}:\Sigma\rightarrow\mathcal{C}$ such that $\tau\circ\mathbf{c}=\bfu$.  Moreover, every non-vertical linear pseudoholomorphic 
 curve in $\mathcal{C}$ is the lift of an isotropic minimal surface in $\mathcal{S}^6$.
\end{cor}

\begin{note} Combining Theorem \ref{ruledthm}, Proposition \ref{Chenlinearprop} and Corollary \ref{Chenlinearcor} we deduce 
Theorem \ref{mainChenthm}.
\end{note}

The horizontal and vertical pseudoholomorphic curves in $\mathcal{C}$, as 
remarked in Definition \ref{horverdfn}, can be regarded as holomorphic curves and thus admit \emph{Weierstrass representations}; 
i.e.~they can be described using holomorphic data.  In contrast, isotropic minimal 
surfaces in $\mathcal{S}^6$ include the pseudoholomorphic curves in $\mathcal{S}^6$ which do \emph{not} admit a Weierstrass
 representation. 
Therefore linear pseudoholomorphic curves 
in $\mathcal{C}$, and thus associative submanifolds in $\mathcal{S}^7$ satisfying Chen's equality, cannot be described purely 
using holomorphic data.  

 We also observe that not every lift of an isotropic minimal  surface in $\mathcal{S}^6$ is necessarily linear since, naively, the pseudoholomorphic lifts to $\mathcal{C}$ of isotropic minimal surfaces depend on 8 functions of 1 variable locally, whereas the linear pseudoholomorphic curves only depend on 6 functions of 1 variable.

\medskip
 
We conclude this section with the following observation concerning \emph{austere} submanifolds in $\mathcal{S}^7$, which we 
 first define. 

\begin{dfn}
Let $A$ be a 3-dimensional submanifold of $\mathcal{S}^{n}$ and let $\II_A$ be its second fundamental form.  We  
call $A$ \emph{austere} if, for all unit vectors $\bff\in N_pA$, the quadratic form $\bff\cdot\II_A$ 
 on $T_pA$ has eigenvalues $\{0,\pm\mu\}$ for some $\mu$.  
\end{dfn}

By \cite[Theorem 3.17]{HarLaw}, an austere 3-fold $A\subseteq\mathcal{S}^{n}$ 
 has ``twisted normal cone'' 
$$\{(sp,t\bff(p))\in\R^{n+1}\oplus\R^{n+1}\,:\,\,s,t\in\R^+,\,p\in A,\,\bff(p)\in N_pA,\,|\bff(p)|=1\}$$
which is absolutely volume-minimizing.
Moreover, it follows from \cite[Theorem 3.1]{KarigiannisMinoo} that $A\subseteq\mathcal{S}^{n}$ is austere if and only if its conormal bundle in $T^*\mathcal{S}^{n}$ is special Lagrangian with respect to the Stenzel metric.
\begin{prop}
Associative 3-folds in $\mathcal{S}^7$ satisfying Chen's equality are austere.
\end{prop}

\begin{proof}
This is an elementary calculation given the local formula \eq{Chensffeq2} for the second fundamental form of an associative 3-fold satisfying 
Chen's equality.
\end{proof}

\begin{remarks}
In \cite[Proposition 6.18]{LotayLag} the author showed that a Lagrangian in $\mathcal{S}^6$ is 
austere if and only if it satisfies Chen's equality.  It is not yet clear whether or not the same result will hold 
 for austere associative 3-folds in $\mathcal{S}^7$.
\end{remarks}
 
\section[Isometric deformations]{Isometric deformations}\label{isom}

In this section we consider the following natural question: if there exists an immersion  
$\iota$ of a connected Riemannian 3-manifold $(A,g_A)$ in $(\mathcal{S}^7,g)$ as an associative 3-fold such that $\iota^*(g)=g_A$, is 
$\iota$ necessarily unique up to rigid motion?  

Necessary conditions for $\iota$ to exist are given by the Gauss, Codazzi and Ricci equations 
\eq{Gausseq}-\eq{Riccieq}. These conditions are sufficient in the case where $A$ is simply connected and, moreover, there exists a unique isometric immersion $\iota$ up to rigid motion with a given second fundamental form and normal connection.  We can see this through a now standard application of Cartan's lemma on maps into Lie groups as follows.  
 
The metric $g_A$ determines an $\SO(3)$-bundle $E$ over $A$. We can view the data given by the metric, second fundamental form and
 normal connection on $A$ locally as the matrices of 1-forms we denoted $\alpha$, $\beta$ and $\gamma$ defined on $E$.  
Using $\alpha$, $\beta$, $\gamma$ and a local orthonormal coframe given by $\omega$, we may thus define a form $\phi$ on $E$ by 
\eq{phispin7eq}, with $\eta=0$.  The form $\phi$ takes values in 
$\spin7$  
and the structure equations
 on $E$ are equivalent to the condition that $\d\phi+\phi\w\phi=0$.  Applying Cartan's result, for each $q\in E$ there exists an open
 set $V\ni q$ and an $\SO(3)$-equivariant immersion $f:V\rightarrow\Spin(7)$ so that $f^*(\Phi)=\phi$, where $\Phi$ is the
 Maurer--Cartan form on $\Spin(7)$.  Moreover, $f$ is unique up to $\Spin(7)$ transformations.  Projecting $f$ 
to $\mathcal{S}^7\cong\Spin(7)/\GG_2$ and using the $\SO(3)$-equivariance, we deduce that for each $p\in A$ there exists an open
 set $U$ and an immersion of $U$  in $\mathcal{S}^7$, with the given metric, second fundamental form and normal connection,
 which is then uniquely determined up to $\Spin(7)$ motions.  One may extend this local construction to a global one, i.e. we can 
take $V=E$ and $U=A$, in the case where $A$ is simply connected.

The equations \eq{Gausseq}-\eq{Riccieq} impose strong restrictions on $g_A$, the second fundamental form and normal connection, so we would 
expect that just asking for $\iota$ to be an isometric immersion (not specifying the second fundamental form or normal connection) would determine it up to rigid motion.    
However,  we find that if  the 
associative submanifold is \emph{ruled}, then $\iota$ need not be unique.  We therefore restrict our attention to this situation.

\medskip

Recall that ruled associative 3-folds in $\mathcal{S}^7$ are in one-to-one correspondence with pseudoholomorphic curves in 
the space $\mathcal{C}$ of oriented geodesic circles in $\mathcal{S}^7$ by Proposition \ref{ruledprop}.  We characterised the adapted frame bundle $E$ of a pseudoholomorphic curve $\Gamma$ in $\mathcal{C}$ in 
 Proposition \ref{pholostructprop} locally using vectors of  1-forms 
$\mathfrak{h}=(\mathfrak{h}_1,\mathfrak{h}_2,\mathfrak{h}_2)^{\rm T}$ and $\theta=(\theta_1,\theta_2,\theta_2)^{\rm T}$, 
a skew-hermitian matrix of 1-forms $\kappa=(\kappa_{jk})$ and a real 1-form $\rho=i\Tr\kappa$ satisfying the structure equations. We see from \eq{subseq1}-\eq{subseq9} that if $A$ is the ruled associative 3-fold in $\mathcal{S}^7$ corresponding to $\Gamma$, then we can relate $\mathfrak{h}$, $\theta$,  $\kappa$ and $\rho$ to the matrices of 1-forms $\omega$,  $\alpha$, $\beta$, $\gamma$ defined locally on the adapted frame bundle of $A$. 

We deduce that the forms $\mathfrak{h}_1$, $\theta_1$, $\rho$, $\kappa_{11}$, $\kappa_{22}$ and
 $\kappa_{33}$  are determined by $\alpha$ and $\omega$, and thus by the Levi-Civita connection of $A$.  In contrast, $\kappa_{32}$ depends only on $\gamma$, and thus is determined by the components of the normal connection of $A$  which are generically independent of the Levi-Civita connection.
 Finally, we observe that
$\mathfrak{h}_2$, $\theta_2$, $\mathfrak{h}_3$, $\theta_3$ $\kappa_{21}$ and $\kappa_{31}$ depend solely on $\beta$, and thus are determined by the second fundamental form of $A$.

The structure equations \eq{pholostructeq3}-\eq{pholostructeq5} on the frame bundle $E$ 
over a surface $\Gamma$ are necessary conditions for a
pseudoholomorphic immersion of $\Gamma$ in $\mathcal{C}$, 
satisfying \eq{pholostructeq1}-\eq{pholostructeq2}, to exist.  
Thus, these equations \eq{pholostructeq3}-\eq{pholostructeq5} may be interpreted as necessary conditions for a ruled associative immersion to exist 
satisfying the structure equations \eq{poseq}-\eq{normeq}.  Since $\mathcal{C}\cong\Spin(7)/\U(3)$ and we view $\Spin(7)$ as the
 $\U(3)$-frame bundle over $\mathcal{C}$, we can apply Cartan's lemma on maps into Lie groups as explained above to deduce that  
\eq{pholostructeq3}-\eq{pholostructeq5} become sufficient conditions for the existence of 
pseudoholomorphic immersions of open neighbourhoods of each point of $\Gamma$ into
 $\mathcal{C}$ satisfying \eq{pholostructeq1}-\eq{pholostructeq2}.  
  Moreover, $\mathfrak{h}$, $\theta$ and $\kappa$ determine these immersions uniquely up to $\Spin(7)$
 transformations on $\mathcal{C}$.  Again, as for the above discussion with associative 3-folds, these local constructions 
 can be extended to global ones if $\Gamma$ is connected and simply connected.

\medskip

Now suppose that $\Gamma$ is a connected pseudoholomorphic curve in $\mathcal{C}$ and $E$ is the frame bundle over $\Gamma$. 
 We can always adapt frames on $\Gamma$ so that $\mathfrak{h}_3=\theta_3=\kappa_{31}=0$ on $E$, since in choosing our unitary frame for
 $T^{(1,0)}\mathcal{C}|_\Gamma$ we still have the freedom to rotate the frame field normal to $T^{(1,0)}\Gamma$.  
(This adaptation simply makes our calculations easier and is not necessary for the construction we will describe.)
Observe that \eq{pholostructeq3}-\eq{pholostructeq5} are invariant under the 
transformation
\begin{align}
(\mathfrak{h}_2,\theta_2,\kappa_{21})&\mapsto e^{ia}
(\mathfrak{h}_2,\theta_2,\kappa_{21})\label{deformeq1}
\end{align}   
for a real constant $a$.  
Moreover, \eq{pholostructeq3}-\eq{pholostructeq5} are also invariant under
\begin{align}
\kappa_{32}&\mapsto e^{ib}\kappa_{32}\label{deformeq2}
\end{align}
for some real constant $b$.  

\begin{remarks}
 If the projection $\tau(\Gamma)$ in $\mathcal{S}^6$, given in Definition \ref{fibrationdfn}, is a surface then we may essentially identify $\kappa_{21}$ and $\kappa_{32}$ with the second and third fundamental forms of $\tau(\Gamma)$.  
We may therefore regard
 $\kappa_{21}$ and $\kappa_{32}$, in some sense, as the second and third fundamental forms of $\Gamma$.  
With this point of view, \eq{deformeq1}-\eq{deformeq2}  basically correspond to rotations of these fundamental forms.  The first deformation should be familiar from the theory of minimal surfaces in space forms, where one rotates the eigendirections of the second fundamental form along the 1-parameter family of isometric deformations. This can be seen geometrically 
when one transforms from the catenoid to the helicoid, for example.
\end{remarks}

Given a solution to \eq{pholostructeq1}-\eq{pholostructeq5} with $\mathfrak{h}_3=\theta_3=\kappa_{31}=0$ (i.e.~a pseudoholomorphic curve $\Gamma$ in $\mathcal{C}$ with adapted frame bundle $E$ determined by the local data $\mathfrak{h}$, $\theta$ and $\kappa$), 
we can construct solutions to \eq{pholostructeq3}-\eq{pholostructeq5} by 
\begin{gather}
\mathfrak{h}_{(a,b)}=(\mathfrak{h}_1,e^{ia}\mathfrak{h}_2,0),\qquad\theta_{(a,b)}=(\theta_1,e^{ia}\theta_2,0),\label{ldata1}\\[4pt]
\kappa_{(a,b)}=\left(\begin{array}{ccc} \kappa_{11} & -e^{-ia}\overline{\kappa}_{21} & 0\\e^{ia}\kappa_{21} & \kappa_{22} & -e^{-ib}\bar{\kappa}_{32}\\
0 & e^{ib}\kappa_{32} & \kappa_{33} \end{array}\right)\label{ldata2}
\end{gather}
for $a,b\in\R$.   
The parameter $a$ in \eq{deformeq1} gives a non-trivial family of solutions unless $\mathfrak{h}_2=\theta_2=
\kappa_{21}=0$, which corresponds to the ruled associative 3-fold $A$ being totally geodesic.  The parameter $b$ in \eq{deformeq2} 
will also give a non-trivial family of solutions for non-zero $\kappa_{32}$.  
 
 However, it may not be possible to vary the parameters $a$ and $b$ independently.  Since $\kappa_{32}$ is determined by the 1-forms
 $\gamma_1$ and $\gamma_2$ on the adapted frame bundle of $A$, $\kappa_{32}$ is generically independent from the forms given by
 $\omega$, $\alpha$ and $\beta$, and thus from $\mathfrak{h}$, $\theta$ and the other forms in $\kappa$.  Therefore, for generic
 choices of pseudoholomorphic curve in $\mathcal{C}$ (or, equivalently, ruled associative 3-fold in $\mathcal{S}^7$), we have a
 2-parameter family of solutions indexed by $(a,b)\in\R^2$ (or $(e^{ia},e^{ib})\in\mathcal{S}^1\times\mathcal{S}^1=T^2$).  However, it can certainly occur that $\kappa_{32}$ is determined by the
 other forms: for example, if $A$ is Lagrangian then $\kappa_{32}=\theta_1$, and if $A$ is minimal Legendrian then
 $\kappa_{32}=-\mathfrak{h}_1$.  In this situation we will only have at most a 1-parameter family of solutions.

By our earlier discussion, for each $p\in\Gamma$, the solutions to \eq{pholostructeq3}-\eq{pholostructeq5} we have constructed 
using \eq{ldata1}-\eq{ldata2} define a family $\iota_{(a,b)}$ of pseudoholomorphic immersions in $\mathcal{C}$ of an open set 
$U\ni p$ with local data given by $\mathfrak{h}_{(a,b)}$, $\theta_{(a,b)}$ and $\kappa_{(a,b)}$.
   From \eq{ldata1}-\eq{ldata2}, we see that the $\iota_{(a,b)}$ induce the same metric.  Since varying the parameters rotates the 
eigendirections of the second and third fundamental forms, the immersions $\iota_{(a,b)}$ will be non-congruent 
 for $a,b\in[0,2\pi)$ as long as $\kappa_{21}$ and $\kappa_{32}$ are non-constant, since this exceptional case is exactly when 
all directions are eigendirections. 
We conclude that for all $p\in A$ there exist a family $\iota_{(a,b)}$ of isometric associative immersions of 
an open set $U\ni p$, which will be non-congruent in generic situations (again since we are transforming the second 
fundamental form of $A$).

We deduce the following result.  
\begin{thm}\label{localprop}
A generic ruled associative 3-fold in $\mathcal{S}^7$ has a $T^2$-family of local non-congruent 
isometric deformations.  
\end{thm}

\medskip

The big problem in extending the local deformations in Theorem \ref{localprop} to global isometric deformations 
is that the parameters $a,b$ in \eq{deformeq1}-\eq{deformeq2} are only guaranteed to be globally defined on the 
simply connected cover of a connected surface $\Gamma$, as we saw in our earlier discussion.  
The question of when the parameters are well-defined on surfaces of positive genus is very difficult to answer 
in general : when $\Gamma\cong T^2$ this is known as the ``period problem''.  However, we can of course overcome 
this difficulty by considering $\Gamma\cong\mathcal{S}^2$.  In this case, as we have already seen, the structure equations 
\eq{pholostructeq3}-\eq{pholostructeq5} become necessary and sufficient conditions for the existence of a pseudoholomorphic 
immersion of $\Gamma$, and thus of a ruled associative immersion of $(A,g_A)$.  We therefore potentially have a $T^2$-family
 of isometric ruled associative immersions of $(A,g_A)$ in $\mathcal{S}^7$.  

Using the well-developed theory of minimal (or \emph{harmonic}) 2-spheres in $\mathcal{S}^6$ we have the following theorem.

\begin{thm}\label{minhorthm}
Every non-totally geodesic minimal $\mathcal{S}^2$ in $\mathcal{S}^6$ has a horizontal pseudoholomorphic lift to $\mathcal{C}$, in the
 sense of Definition \ref{horverdfn}.
\end{thm}

\begin{proof}  Let $\Sigma$ be a non-totally geodesic minimal 2-sphere in $\mathcal{S}^6$.  By the work in \cite{Calabi1}, 
$\Sigma$ is linearly full in a totally geodesic $\mathcal{S}^{2m}$ for $m=2$ or $3$ and 
$\Sigma$ is superminimal, hence isotropic in the sense of Definition \ref{ellipsedfn}.  An idea due 
to Chern \cite{Chern} is to associate to $\Sigma$ its so-called directrix curve $\Xi$, which is a holomorphic curve in
 $\C\P^{2m}$.  In fact, $\Xi$ is a totally isotropic curve in $\C\P^{2m}$; i.e.~if $\xi:\Xi\rightarrow\C\P^{2m}$ then 
$(\xi,\xi)=(\xi^{\prime},\xi^{\prime})=\ldots=(\xi^{m-1},\xi^{m-1})=0$, where $(\,,\,)$ denotes the complex inner product on
$\C\P^{2m}$.  Thus $\Xi$ defines a curve $\Gamma$ in the space 
$H_m$ of totally isotropic $m$-dimensional subspaces of $\C^{2m+1}$ by sending each point in $\Xi$ to the subspace spanned by 
$\xi$ and its first $(m-1)$ derivatives at that point.  Clearly, $H_m$ is contained in a complex Grassmannian and thus in a complex 
projective space.  It follows from the work in \cite{Barbosa} that $H_m\cong\SO(2m+1)/\U(m)$ and that 
 $\Gamma$ is a holomorphic curve in $H_m$ which is horizontal with respect to the fibration of $\SO(2m+1)/\U(m)$ over 
$\mathcal{S}^{2m}$.  

Thus there is a well-defined horizontal lift, which we also call $\Gamma$, of $\Sigma$ to $\mathcal{C}\cong\Spin(7)/\U(3)$ which 
is \emph{holomorphic} with respect to the integrable complex structure $I_{\mathcal{C}}$ 
on $\mathcal{C}$ introduced in Definition \ref{Cdfn}.  
 However, by \cite[Lemma 5.2 \& Corollary 5.3]{Fox}, the complex structures $J_{\mathcal{C}}$ and $I_{\mathcal{C}}$ on
$\mathcal{C}$ agree on the 
horizontal distribution $\mathcal{H}$ given in Definition \ref{horverdfn}, and so $\Gamma$ is pseudoholomorphic with respect to $J_{\mathcal{C}}$. 
\end{proof}

\noindent We can think of the horizontal pseudoholomorphic lift to $\mathcal{C}$ of a minimal 2-sphere in $\mathcal{S}^6$ as a ``twistor lift''.

\begin{remark}
The proof of Theorem \ref{minhorthm} can clearly be generalised to show that any \emph{superminimal} surface in $\mathcal{S}^6$ has a horizontal pseudoholomorphic lift to $\mathcal{C}$.
\end{remark}

From Theorem \ref{minhorthm}, we have a horizontal pseudoholomorphic immersion of $\Gamma\cong\mathcal{S}^2$ in $\mathcal{C}$ which 
then defines local data $\mathfrak{h}$, $\theta$ and $\kappa$ with $\mathfrak{h}=0$.  Since $\Gamma$ is non-totally geodesic, connected and simply
 connected, we have that $\kappa_{21}\neq 0$ and so the parameter $a$ in \eq{deformeq1} defines an $\mathcal{S}^1$-family 
of immersions of $\Gamma$ which all induce the same metric.  Moreover, these immersions will be non-congruent if 
$\Gamma$ does not have constant curvature.  We deduce our main result.

\begin{thm}\label{isommainthm}
Let $\bfu:\mathcal{S}^2\rightarrow\mathcal{S}^6$ be non-totally geodesic and minimal, and recall Example \ref{ruledex} and 
Definition \ref{horverdfn}.  Let $\mathbf{c}_\Gamma:\mathcal{S}^2\rightarrow\mathcal{C}$ be a horizontal pseudoholomorphic lift
 defined by a holomorphic curve $\Gamma$ in $\mathcal{X}(\mathcal{S}^2)$ and let $A=A(\mathcal{S}^2,\bfu,\Gamma)$.  

There is an $\mathcal{S}^1$-family of isometric associative deformations of $A$ in $\mathcal{S}^7$.  Moreover, the family consists of
 non-congruent associative 3-folds if $\bfu:\mathcal{S}^2\rightarrow\mathcal{S}^6$ does not have constant curvature.
\end{thm}

\noindent We cannot use the parameter $b$ in this case to give a second $\mathcal{S}^1$-family of deformations of $A$
because one may calculate that $\kappa_{32}=0$ for the horizontal pseudoholomorphic lifts of minimal 2-spheres.

\begin{note}
From Proposition \ref{Chenlinearprop} and Theorems \ref{minhorthm} and \ref{isommainthm} we deduce Theorem \ref{isomthm}.
\end{note}

By \cite[Corollary 4.14]{Barbosa}, the area of a minimal 2-sphere $\Sigma$ in $\mathcal{S}^6$ is always an integer multiple of 
$4\pi$.  Thus, we can define the \emph{degree} $d_\Sigma$ of $\Sigma$ by Area$(\Sigma)=4\pi d_\Sigma$.  This is the same 
as the degree of the associated twistor lift: the horizontal pseudoholmorphic curve in $\mathcal{C}$ viewed 
as a holomorphic curve in a complex projective space as in \cite{Barbosa}.  As already observed, by the work of 
Calabi \cite{Calabi1}, a minimal 2-sphere in $\mathcal{S}^n$ must be linearly full in an even-dimensional 
totally geodesic hypersphere.   
The main results of \cite{Fernandez} and \cite{Loo} show that the moduli space of minimal 2-spheres of degree $d$ which are 
full in $\mathcal{S}^{2m}$ for $m=2$ or $3$ has dimension $2d+m^2$.  Moreover, by \cite[Theorem 6.19]{Barbosa}, every 
integer $d\geq 6$ is the degree of some minimal 2-sphere in $\mathcal{S}^6$, so there is an arbitrarily large moduli space of 
 minimal 2-spheres with the same area.  These spheres will give rise to 1-parameter families of isometric
 associative immersions by Theorem \ref{isommainthm}.
 
We now study Theorem \ref{isommainthm} in some special cases.

\begin{cor}\label{minLegcor}
Given any non-constant curvature minimal $\bfu:\mathcal{S}^2\rightarrow\mathcal{S}^4$, there exists an $\mathcal{S}^1$-family of
 non-congruent isometric minimal Legendrian submanifolds in $\mathcal{S}^7$ satisfying Chen's inequality which are ruled over
 $\mathcal{S}^2$. 
\end{cor}

\begin{proof}
By Proposition \ref{varpiprop}(b) and its proof, we see that a horizontal
 lift of a minimal $\mathcal{S}^2$ in $\mathcal{S}^4\subseteq\mathcal{S}^6$
 will define a minimal Legendrian submanifold $A$ of $\mathcal{S}^7$ ruled over
 $\mathcal{S}^2$.  By Theorem \ref{isommainthm} we have an $\mathcal{S}^1$-family of isometric associative 3-folds containing $A$.  However, since the transformation \eq{deformeq1} leaves the conditions given in Example \ref{minLegsex} invariant, this family will consist of minimal Legendrian submanifolds.
\end{proof}

In a similar manner, using Example \ref{Lagsex}, Proposition \ref{varpiprop}(c) and Theorem \ref{isommainthm} we have the following result.

\begin{cor}
Given any non-constant curvature pseudoholomorphic $\bfu:\mathcal{S}^2\rightarrow\mathcal{S}^6$,
there exists an $\mathcal{S}^1$-family of non-congruent isometric Lagrangian submanifolds in $\mathcal{S}^6$ satisfying Chen's equality which are ruled over $\mathcal{S}^2$.
\end{cor}

\begin{remark}
One can see this result directly from the structure equations for Lagrangians satisfying Chen's equality given in \cite[$\S$6.4]{LotayLag}.  The Lagrangians are tubes of radius $\frac{\pi}{2}$ in the second normal bundle of 
the pseudoholomorphic $\mathcal{S}^2$.
\end{remark}

As a final comment, it would be interesting to know whether a minimal 2-sphere could be lifted to a pseudoholomorphic curve 
so that the parameter $b$ in \eq{deformeq2} defines a second independent $\mathcal{S}^1$-family of deformations. 
Necessarily such a 2-sphere would have to be linearly full in
 $\mathcal{S}^6$ but not pseudoholomorphic.  Given this minimal $\mathcal{S}^2$ and its pseudoholomorphic lift, we would then be able
 to define a 2-torus family of isometric associative embeddings using \eq{deformeq1}-\eq{deformeq2}.

\begin{ack}
The author is indebted to Robert Bryant for enlightening discussions, insight and helpful advice.  
He would also like to thank Daniel Fox and Mark Haskins for useful conversations and comments, and University College, Oxford and Imperial College, London for hospitality during the course of this project.  

The author is supported by an EPSRC Career Acceleration Fellowship.  
\end{ack}



\end{document}